\documentclass[12pt]{amsart}

\usepackage{a4wide}
\usepackage{amssymb,amsmath,amsthm}
\usepackage{multirow}
\usepackage{epsfig}
\usepackage{pst-node}
\usepackage{graphicx}
\usepackage{enumitem}
\usepackage{color}
\usepackage[pdfauthor={Jozef H. Przytycki,Krzysztof K. Putyra},pdftitle={Homology of distributive lattices},colorlinks]{hyperref}

\definecolor{internalLink}{rgb}{0.5,0,0}
\definecolor{citeLink}{rgb}{0,0.5,0}
\definecolor{urlLink}{rgb}{0,0,0.5}

\hypersetup{linkcolor=internalLink}
\hypersetup{citecolor=citeLink}
\hypersetup{urlcolor=urlLink}

\newtheorem{theorem}{Theorem}[section]
\newtheorem{lemma}[theorem]{Lemma}
\newtheorem{remark}[theorem]{Remark}
\newtheorem{definition}[theorem]{Definition}
\newtheorem{conjecture}[theorem]{Conjecture}
\newtheorem{proposition}[theorem]{Proposition}
\newtheorem{corollary}[theorem]{Corollary}

\newtheorem{algorithm}[theorem]{Algorithm}

\long\def\symbolfootnote[#1]#2{\begingroup%
\def\thefootnote{\fnsymbol{footnote}}\footnote[#1]{#2}\endgroup} 

\DeclareMathOperator{\id}{id}
\DeclareMathOperator{\rk}{rk}

\DeclareMathOperator{\Z}{\mathbb{Z}}
\DeclareMathOperator{\F}{\mathcal{F}^0}
\DeclareMathOperator{\R}{\mathcal{R}}
\DeclareMathOperator{\orb}{\mathcal{O}}
\newcommand{\CF}{\left(C/\!\mathcal{F}^0\right)}
\DeclareMathOperator{\SetS}{\mathcal{S}}
\DeclareMathOperator{\lt}{\vdash}
\DeclareMathOperator{\rt}{\dashv}
\newcommand{\triangleup}{%
	\mathrel{\reflectbox{\rotatebox[origin=c]{180}{$\triangledown$}}}}

\newcommand{\projleft}{right trivial}

\newcommand{\projright}{left trivial}

\newcounter{countcomments}
\newcommand{\hide}[1]{}		% Hides unnecessary sections

\input diagxy

\begin{document}

\title{Homology of distributive lattices}

\author[J\'ozef H. Przytycki]{J\'ozef H. Przytycki}
\thanks{JHP was partially supported by the~NSA-AMS 091111 grant,
by the~Polish Scientific Grant: Nr. N-N201387034, and by the~GWU REF grant.}
\address{Department of Mathematics, George Washington University\\ Washington, DC 20052\\
and Gda\'nsk University, Poland}
\email{przytyck@gwu.edu}

\author[Krzysztof K. Putyra]{Krzysztof K. Putyra}
\thanks{KP was supported by the NSF grant DMS-1005750 in summer 2011.\\
	\indent ${}^1$ Pawe\l\ Waszkiewicz (1973-2011) was a~faculty member of Theoretical Computer Science at
	Jagiellonian University in Krakow. He obtained PhD at the~University of Birmingham,
	UK, in 2002 in the~theory of domains and formal languages. Although his career
	has been ceased in a~tragic way nine years later, he has already published 21 papers.
	The~second author is indebted to him for being introduced
	to the~fascinating world of categories, posets and domains.
}
\address{Department of Mathematics, Columbia University\\ New York, NY 10027}
\email{putyra@math.columbia.edu}

\begin{abstract}
We outline the theory of sets with distributive operations: multishelves
and multispindles, with examples provided by semi-lattices, lattices and skew lattices.
For every such a structure we define multi-term distributive homology
and show some of its properties.
The main result is a~complete formula for the~homology of
a~finite distributive lattice. We also indicate the~answer for unital
spindles and conjecture the~general formula for semi-lattices and some skew lattices.
Then we propose a~generalization of a~lattice as a~set with a~number of
idempotent operations satisfying the~absorption law.
\end{abstract}

%\date{August 2011}

\dedicatory{The paper is dedicated to Pawe\l\ Waszkiewicz${}^1$}

\maketitle

\tableofcontents

\stepcounter{footnote}

\section{Introduction}\label{chpt:intro}
While homology of associative structures (e.g. groups or rings) have been studied
successfully for a~long time, homology theory of distributive structures started
to develop only recently. The~homology theory of racks (i.e. sets with a~right
self-distributive invertible binary operation) was introduced about 1990 by~Fenn,
Rourke and Sanderson \cite{Fenn,FennRourkeSand} in relation to higher dimensional
knot theory. The first full calculation of rack homology was that of prime dihedral
quandles \cite{N-P-2,Nos,Cla}. In this paper we outline the~general theory
of multishelves and distributive homology. It is a~new discipline on the~border
of algebra and topology with intention to be comparable with homological algebra
of associative structures. The main result of the~paper, the~theorem~\ref{thm:hom-all},
gives a~complete determination of the~structure of a~multiterm homology of
a~finite distributive lattice. In particular, it solves Conjecture 29 of \cite{P-S}.
 
In the~second section we introduce the concept of a~monoid $Bin(X)$ of binary
operations on $X$ and prove its basic properties related to distributivity.
We define multishelves and multispindles and discuss the~premiere example of them
coming from Boolean algebras and distributive lattices.

In the~third section we define homology of multishelves and introduce the~notion
of a~weak simplicial module, which provides a~good abstract language to discuss
this homology. In particular, we define for any weak simplicial module a~chain
complex of degenerate elements and its natural filtration.

In the~fourth section, for any multispindle we split its homology into degenerate
and normalized parts. We also show another decomposition, very useful to study
homology of distributive lattices, into homology of a~point, reduced initially
degenerate homology, and a~reduced initially normalized part. We discuss basic
properties of them.  

In the~fifth section we show that the~normalized degenerate part can be obtained
from the~early normalized part of the~homology. In the second part of the~section
we completely determine homology of a~finite distributive lattice by first computing
it for the~two element Boolean algebra $B_1$ and then proving Mayer-Vietoris type
of results allowing computing homology of any distributive lattice from 
its proper sublattices.

In the~sixth section we analyze various generalizations of distributive lattices
to which our theory applies fully or partially. In particular, we analyze skew
lattices and introduce the notion of a~generalized distributive lattice of
any number of operations. We formulate several conjectures and support them by 
empirical calculation.

\section{Distributive structures}\label{chpt:monoids}
This section is devoted to establish the~notation.
After stating definitions and properties of the~most basic distributive
structures (shelves, spindles, etc.) we provide classical
examples of such structures: distributive lattices and Boolean algebras.

\subsection{Distributive operations}

Let $X$ be a~set and $\star\colon X\times X \to X$ a~binary operation.
We call a~pair $(X,\star)$ a~\emph{magma}.
Denote by $Bin(X)$ the~set of all binary operations on $X$.
An~easy calculation shows that it has a~monoid structure.

\begin{proposition}\label{prop:Bin-monoid}
$Bin(X)$ is a~monoid (e.g. a~semigroup with a~unit) with a~composition 
given by $x \star_{\!1}\!\!\star_2\, y = (x\star_1 y)\star_2 y$ and the~two-sided unit $\lt$
being the~\projleft{} operation (that is $x\lt y = x$ for any $x,y\in X$).
\end{proposition}

Recall that an~operation $\star_1$ is called \emph{right distributive} with respect to
$\star_2$, if for all $x,y,z\in X$ it satisfies 
\begin{equation}
(x \star_2 y)\star_1 z = (x \star_1 z)\star_2(y \star_1 z).
\end{equation}
Dually we define left distributivity. An~operation $\star$ is called right (resp. left)
\emph{self-distributive}, if it is right (resp. left) distributive with respect
to itself. The~operation $\lt$ is right distributive with respect to any
other operation and vice versa. This plays later an important role.%
\footnote{
	$\lt$ and $\star$ are seldom associative,
	as $(x\lt y)\star z = x\star z$, but $x \lt (y \star z)=x$.
}
Since now by distributivity we will always mean right distributivity.

While an~associative magma is called a~semigroup for a~long time,
the~self-distributive magma did not have an~established name,
even though C.S.~Peirce considered it back in 1880 \cite{Pei}.
Alissa Crans in her PhD thesis \cite{Cr}
suggested the~name a~\emph{right shelf} (or simply a~\emph{shelf}).
Below we write a~formal definition of a~shelf and related notions
of a~spindle, a~rack, and a~quandle.

\begin{definition}\label{def:shelf}
A~magma $(X,\star)$ is called a~\emph{shelf} if $\star$ is self-distributive.
Moreover, if $\star$ is idempotent (i.e.~$x\star x = x$ for any $x\in X$),
then $(X,\star)$ is called a~\emph{spindle} (again a~term coined by Crans).
\end{definition}

\begin{remark}
Early examples of shelves in topology date to J.H.~Conway and D.~Joyce.
In 1959 Conway coined a~name \emph{wrack}, modified later to \emph{rack} \cite{F-R},
for a~shelf with an~invertible product (i.e. $\star_b(x) = x\star b$ is a~bijection
for any $b\in X$). Later Joyce in his PhD thesis \cite{Joy} in 1979 introduced a~\emph{quandle}
as a~rack with an~idempotent product. Axioms of a~quandle were motivated by
the~Reidemeister moves: idempotency by the~first move, invertibility by the~second
and right self-distributivity by the~third move.
\end{remark}

The~above definition describes properties of an~individual magma $(X,\star)$.
It is also useful to consider subsets or submonoids of $Bin(X)$ satisfying related
conditions as described below.

\begin{definition}\label{def:monoid-in-Bin}
A~subset $\SetS\subset Bin(X)$ is called \emph{distributive} if all pairs
$\star_{\alpha},\star_{\beta} \in \SetS$ are mutually right distributive. In particular,
taking $\star_\alpha=\star_\beta$, all operations must be self-distributive.
If in addition $\SetS$ is a~submonoid of $Bin(X)$, we call it a~\emph{distributive submonoid}.
\end{definition}

Any set $\SetS$ generates a~submonoid $M(\SetS)$ in $Bin(X)$.
It is easy to check that if $\SetS$ is a~distributive set,
then $M(\SetS)$ is a~distributive submonoid (see \cite{Prz-Demonstratio}).

\begin{definition}\label{def:multishelf}
A~pair $(X,\{\star_\lambda\}_{\lambda\in\Lambda})$ is called a~\emph{multishelf},
if operations $\star_\lambda$ form a~distributive set in $Bin(X)$.
Furthermore, if each $\star_\lambda$ satisfies the~idempotency condition, we call
$(X,\{\star_\lambda\}_{\lambda\in\Lambda})$ a~\emph{multispindle}.
\end{definition}

\begin{remark}\label{rmk:adding-trivials}
If $(X,\SetS)$ is a~multishelf, then $S\cup\{\lt\}$ is also a~multishelf.
In addition, if $(X,\SetS)$ is a~multispindle, we can enlarge $\SetS$
by the~\projright{} operation~$\rt$ (i.e. $x\rt y = y$).
Furthermore, if $\SetS$ is a monoid and it consists of idempotents,
then it remains a~monoid after enlarging by $\rt$, since it is a~two-sided
projector  (i.e. $\star\!\rt=\rt=\rt\!\star$).
\end{remark}

Finally, we give definitions of substructures and homomorphisms.
All of them are very natural.

\begin{definition}\label{def:substr}
Let $(X,\{\star_\lambda\}_{\lambda\in\Lambda})$ be a~multishelf (resp. multispindle).
A~subset $Y\subset X$ is called a~\emph{submultishelf} (resp. \emph{submultispindle})
if it is closed under all operations $\star_\lambda$.
\end{definition}

\begin{definition}\label{def:shelf-homom}
Let $(X,\{\star_\lambda\}_{\lambda\in\Lambda})$ and $(Y,\{*_\lambda\}_{\lambda\in\Lambda})$
be multishelves (resp. multispindles) with operations indexed with the~same set.
A~map of sets $\varphi\colon X\to Y$ is called a~\emph{multishelf homomorphism}
(resp. a~\emph{multispindle homomorphism}) if it preserves all operations:
$\varphi(x\star_\lambda y) = \varphi(x)*_\lambda\varphi(y)$ for every $\lambda\in\Lambda$.
\end{definition}

The~basic example of a~shelf homomorphism is given by an~action of a~fixed
element $a\in X$. Namely, the~map $\star_\lambda^a(x) := x\star_\lambda a$
is a~shelf homomorphism due to distributivity. More generally, we can put
$$
\star^{a_1,\dots,a_s}_{\lambda_1,\dots,\lambda_s}(x) := ((x\star_{\lambda_1}a_1)\star_{\lambda_2}\cdots)\star_{\lambda_s} a_s
$$
for the~composition of $\star_{\lambda_i}^{a_i}$.
Such homomorphisms form a~monoid%
\footnote{
	By a~convention, the~identity map is given by empty sequences:
	$\id = \star^{\emptyset}_{\emptyset}$.
}
and are called \emph{inner endomorphisms} of a~multishelf $X$.
Similar definitions can be stated for multiracks and multiquandles,
but we omit them since these notions are not used in the~paper.

\subsection{Lattices and Boolean algebras}
Natural examples of multispindles are provided by distributive
lattices and Boolean algebras. Because these examples are very important for this paper,
we include here a~brief introduction to the~theory of lattices.
For a~more detailed course the~reader is referred to~\cite{Gratzer,Sikorski,Traczyk}.

\begin{definition}\label{def:semilattice}
A~magma $(L,\star)$ is called a~\emph{semilattice} if $\star$
is idempotent, commutative and associative.
\end{definition}

\noindent It is didactic to see that the~conditions above imply self-distributivity of $\star$.%
\footnote{
	The~reader should not confuse this with a~notion of a~distributive semilattice,
	which is much stronger (see \cite{Gratzer}, p.~117).
}
The~short proof is given below:
\begin{align*}
(x \star y)\star z &= (x \star y)\star(z \star z) = (x\star(y\star(z\star z))) =\\
	&= (x\star(z\star(y\star z))) = (x\star z)\star(y\star z).
\end{align*}

Any~semilattice is a~partially ordered set (a~\emph{poset}) with the~order
defined as
\begin{equation}
x\leqslant y \quad\equiv\quad x\star y = y.
\end{equation}
Dually, any~poset with unique binary maxima is a~semilattice
with $\star$ given by the~maximum. This allows us to represent semi-lattices
pictorially by \emph{Hasse diagrams} as in fig.~\ref{fig:semi-lattice}.

\begin{figure}[tb]
	\begin{center}
		\hfill\hfill
\begin{pspicture}(4,2)
\cnode*(0,0){2pt}{b1}
\cnode*(2,0){2pt}{b2}
\cnode*(4,0){2pt}{b3}
\cnode*(1,1){2pt}{m1}
\cnode*(2,1){2pt}{m2}
\cnode*(3,1){2pt}{m3}
\cnode*(2,2){2pt}{t}
\ncline[nodesep=2pt]{b1}{m1}
\ncline[nodesep=2pt]{b2}{m1}
\ncline[nodesep=2pt]{b2}{m2}
\ncline[nodesep=2pt]{b2}{m3}
\ncline[nodesep=2pt]{b3}{m3}
\ncline[nodesep=2pt]{m1}{t}
\ncline[nodesep=2pt]{m2}{t}
\ncline[nodesep=2pt]{m3}{t}
\end{pspicture}
\hfill
\begin{pspicture}(3,2)
\cnode*(0.4,0){2pt}{b1}
\cnode*(2,0){2pt}{b2}
\cnode*(1.2,1){2pt}{l1}
\cnode*(2.8,0.6){2pt}{r1}
\cnode*(2.8,1.5){2pt}{r2}
\cnode*(2,2){2pt}{t}
\ncline[nodesep=2pt]{b1}{l1}
\ncline[nodesep=2pt]{b2}{l1}
\ncline[nodesep=2pt]{b2}{r1}
\ncline[nodesep=2pt]{r1}{r2}
\ncline[nodesep=2pt]{l1}{t}
\ncline[nodesep=2pt]{r2}{t}
\end{pspicture}
\hfill
\begin{pspicture}(2,2)
\cnode*(0.2,0){2pt}{b1}
\cnode*(1.8,0){2pt}{b2}
\cnode*(1,1){2pt}{m}
\cnode*(1,2){2pt}{t}
\ncline[nodesep=2pt]{b1}{m}
\ncline[nodesep=2pt]{b2}{m}
\ncline[nodesep=2pt]{m}{t}
\end{pspicture}
\hfill\hfill\ 
	\end{center}
	\caption{Typical Hasse diagrams of semi-lattices. The~right one is a~rooted tree.}
	\label{fig:semi-lattice}
\end{figure}
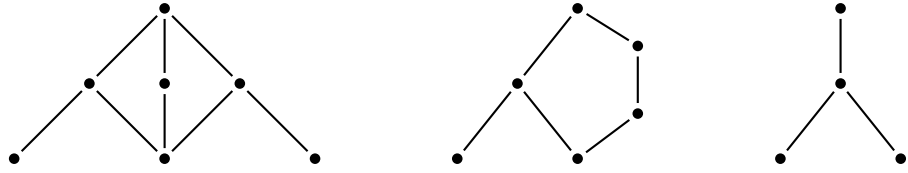

A~typical example of a~semi-lattice is given by a~rooted tree $T$.
Having two vertices $x, y\in T$, we define $x\star y\in T$ as the~first common
point of the~unique paths starting at $x$ or $y$ and ending at the~root of $T$.

Having a~poset, we can ask not only for maxima, but also for minima
and relations those two operations satisfy. If both binary maxima and minima exist
and are unique, the~poset is called a~\emph{lattice}. Below we present an~axiomatic
definition.

\begin{definition}\label{def:lattice}
A~triple $(L,\lor,\land)$ is called a~\emph{lattice},
if both operations are idempotent, commutative, associative
and the~following absorption laws hold:
\begin{equation}\label{eq:absorption}
(x\lor y)\land y = y = (x\land y)\lor y.
\end{equation}
\end{definition}

\noindent The~operation $\lor$ is called a~\emph{join} or \emph{supremum},
whereas $\land$ is called a~\emph{meet} or \emph{infimum}.
Both are self-distributive, but they may not be mutually distributive.
If this is the~case, the~lattice $L$ is called \emph{distributive}.

As in the~case of semilattices, every lattice is a~poset with an~order
defined as
\begin{equation}
x\leqslant y \quad\equiv\quad x\land y = x \quad\equiv\quad x\lor y = y.
\end{equation}
If the~order is linear, a~lattice $L$ is called a~\emph{chain}.

\begin{figure}[b]
	\begin{center}
		\hfill\hfill
\begin{pspicture}(2,3)
\cnode*(1,0){2pt}{b}
\cnode*(0,1){2pt}{m11}
\cnode*(1,1){2pt}{m12}
\cnode*(2,1){2pt}{m13}
\cnode*(0,2){2pt}{m21}
\cnode*(1,2){2pt}{m22}
\cnode*(2,2){2pt}{m23}
\cnode*(1,3){2pt}{t}
\ncline[nodesep=2pt]{b}{m11}
\ncline[nodesep=2pt]{b}{m12}
\ncline[nodesep=2pt]{b}{m13}
\ncline[nodesep=2pt]{m11}{m21}
\ncline[nodesep=2pt]{m11}{m22}
\ncline[nodesep=2pt]{m12}{m21}
\ncline[nodesep=2pt]{m12}{m23}
\ncline[nodesep=2pt]{m13}{m22}
\ncline[nodesep=2pt]{m13}{m23}
\ncline[nodesep=2pt]{m21}{t}
\ncline[nodesep=2pt]{m22}{t}
\ncline[nodesep=2pt]{m23}{t}
\end{pspicture}
\hfill
\begin{pspicture}(1,3)
\cnode*(0.5,0){2pt}{b}
\cnode*(0.5,1){2pt}{m1}
\cnode*(0.5,2){2pt}{m2}
\cnode*(0.5,3){2pt}{t}
\ncline[nodesep=2pt]{b}{m1}
\ncline[nodesep=2pt]{m1}{m2}
\ncline[nodesep=2pt]{m2}{t}
\end{pspicture}
\hfill
\begin{pspicture}(3,3)
\cnode*(1.5 ,0   ){2pt}{b}
\cnode*(0.75,0.75){2pt}{m11}
\cnode*(2.25,0.75){2pt}{m12}
\cnode*(0   ,1.5 ){2pt}{m21}
\cnode*(1.5 ,1.5 ){2pt}{m22}
\cnode*(3   ,1.5 ){2pt}{m23}
\cnode*(0.75,2.25){2pt}{m31}
\cnode*(2.25,2.25){2pt}{m32}
\cnode*(1.5 ,3   ){2pt}{t}
\ncline[nodesep=2pt]{b}{m11}
\ncline[nodesep=2pt]{b}{m12}
\ncline[nodesep=2pt]{m11}{m21}
\ncline[nodesep=2pt]{m11}{m22}
\ncline[nodesep=2pt]{m12}{m22}
\ncline[nodesep=2pt]{m12}{m23}
\ncline[nodesep=2pt]{m21}{m31}
\ncline[nodesep=2pt]{m22}{m31}
\ncline[nodesep=2pt]{m22}{m32}
\ncline[nodesep=2pt]{m23}{m32}
\ncline[nodesep=2pt]{m31}{t}
\ncline[nodesep=2pt]{m32}{t}
\end{pspicture}
\hfill
\begin{pspicture}(2,3)
\cnode*(1.2,0){2pt}{b}
\cnode*(0,1.5){2pt}{l1}
\cnode*(2,0.9){2pt}{r1}
\cnode*(2,2.1){2pt}{r2}
\cnode*(1.2,3){2pt}{t}
\ncline[nodesep=2pt]{b}{l1}
\ncline[nodesep=2pt]{b}{r1}
\ncline[nodesep=2pt]{r1}{r2}
\ncline[nodesep=2pt]{l1}{t}
\ncline[nodesep=2pt]{r2}{t}
\end{pspicture}
\hfill\hfill\ 
	\end{center}
	\caption{Examples of lattices: a~Boolean algebra $B_3$, a~chain lattice $L_4$,
				a~distributive lattice $L_3\times L_3$ and a~non-distributive
				lattice $N_5$.}
	\label{fig:lattice}
\end{figure}
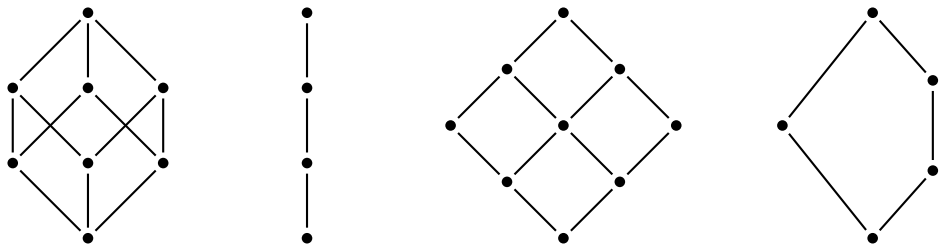

Join and meet operations provide existence of infima and suprema of finite
subsets. However, general infimima and suprema might not exist. If they do,
the~lattice $L$ is called \emph{complete}. The~two special elements
$\bot = \inf L$ and $\top = \sup L$ are called \emph{bottom} and \emph{top}.
Clearly, any~finite lattice is complete.

A~special case of a~lattice is a~\emph{Boolean algebra}.
Although an~axiomatic definition exists, for our purpose it is enough
to know that all finite Boolean algebras are modeled as families of all subsets
of a~given finite set $A$. We write $B_n:=\mathcal{P}(A)$ for
a~Boolean algebra modeled on a~set $A$ with $n$ elements.

Similarly as for multishelves, we can define a~sublattice.
It is worth to mention that if a~lattice $L$ is distributive,
then any sublattice $L'\subset L$ is distributive as well.
Two kinds of sublattices are distinguished due to their special properties.

\begin{definition}\label{def:ideal-filter}
Let $L$ be a~lattice.
\begin{enumerate}[label={\normalfont(\roman{*})}]
\item An~\emph{ideal} $I\subset L$ is a~sublattice such that $I\land L\subset I$
\item A~\emph{filter} $F\subset L$ is a~sublattice such that $F\lor L\subset F$
\item An~ideal (filter) is called \emph{principal} if it is generated by one element,
		i.e. $I=L\land x$ (resp. $F = L\lor x$). In this case we write
		$I = \downarrow\!\! x$ (resp. $F = \uparrow\!\! x$).
\end{enumerate}
\end{definition}

\begin{remark}
If a~lattice $L$ is finite, then all ideals and filters are principal.
Namely, for any ideal $I$ and filter $F$ one has $I = \downarrow\!\max\!I$
and $F=\uparrow\!\min\!F$.
\end{remark}

In the~order-theoretic approach, an~ideal is a~subset closed under suprema
and containing with any $a\in I$ all elements smaller than $a$.
Dually, a~filter is a~subset closed under infima and containing with any $a\in F$
all elements greater than $a$.

The~last notion we need is irreducibility.
An~element $x\in L$ is called \emph{join-irreducible} (or simply \emph{irreducible})
if for any decomposition $x = a\lor b$ we have either $x=a$ or $x=b$.
Clearly, the~bottom element $\bot$, if it exists, is irreducible.
The~set of all non-minimal irreducible elements in $L$ is denoted by $J(L)$,
with the~letter $J$ reserved for the~cardinality of this set.
Dually we define \emph{meet-irreducible} elements.
The~following theorem combines irreducible elements with the~size of a~lattice
(see~\cite{Gratzer}, cor. 7.14).

\begin{theorem}\label{thm:irred-chain}
Let $L$ be a~finite distributive lattice and $\mathcal{C}\subset L$ a~maximal chain in $L$
of length $n = |\mathcal{C}|-1$. Then $|J(L)| = n$.
\end{theorem}

Although a~lattice $L$ is defined as a~set with two operations,
it has four operations as a~multispindle due to remark~\ref{rmk:adding-trivials}.
Furthermore, the~absorption law (\ref{eq:absorption})
provides the~two compositions are equal to the~\projright{} operation
$x \rt y = y$. It is worth to mention that the~set $\{\lt,\lor,\land,\rt\}$
is a~distributive submonoid of $Bin(L)$.
Similarly, a~semilattice $(X,\star)$ as a~multispindle has three operations
$\{\lt,\star,\rt\}$ which form a~distributive submonoid of $Bin(X)$.

\section{Multi-term chain complexes}\label{chpt:complex}
For any distributive structure we can create a~chain complex
and compute its homology. The~construction is given below,
followed by a~short introduction to the~theory of weak simplicial
modules. We decided to include this theory, because it explains
why some of the~decompositions described in section~\ref{chpt:decomp} exist.

\subsection{Homology of distributive structures}

Let $(X,\star)$ be a~shelf and $R$ any commutative ring.
We define a~(one-term) distributive chain complex
$C^\star(X;R)$ as follows:
\begin{align*}
	C^\star_n(X;R)&:=R X^{n+1} = R\langle (x_0,\dots,x_n) | x_i\in X\rangle, \\
	\partial_n^{\star}&:= \sum_{i=0}^n (-1)^i d_i^{\star},
\end{align*}
where maps $d_i^\star$ are given by the~formula
\begin{align*}
	d_0^\star(x_0,...,x_n) &= (x_1,...,x_n) \\
	d_i^\star(x_0,...,x_n) &= (x_0\star x_i,...,x_{i-1}\star x_i,x_{i+1},...,x_n)
\end{align*}
We check that $\partial^{\star}\partial^{\star}=0$. The~homology of this chain complex is called 
a~\emph{one-term distributive homology} of $(X,\star)$ with coefficients in $R$
and is denoted by $H^\star(X;R)$.
We can put $C_{-1}=R$ and $\partial_0(x)=1$ to obtain \emph{augmented} distributive
chain complex and augmented distributive homology $\tilde H^\star(X;R)$.
As in the~classical case we get:
$$
H_n^\star(X;R)=\begin{cases}
	R\oplus \tilde H^\star_n(X;R),  & n = 0, \\
	\tilde H^\star_n(X;R),          & n > 0.
\end{cases}
$$

The~first homology theory related to self-distributive structures was constructed in 1990 by 
Fenn, Rourke and Sanderson \cite{Fenn,FennRourkeSand} and motivated by higher dimensional
knot theory.  For a~rack $(X,\star)$ they defined rack homology $H_n^{\R}(X)$ by taking a~chain
complex $C^{\R}_n(X)=\Z X^n$ with $\partial^{\R}_n\colon C^{\R}_n(X)\to C^{\R}_{n-1}(X)$ given by
$\partial_n^{\R} = \partial_{n-1}^\star-\partial_{n-1}^{\lt}$.%
\footnote{
	Our notation has grading shifted by 1, that is $C_n(X) = C^{\R}_{n+1}(X) = \Z X^{n+1}$.
}
This is a~routine check that $\partial^{\R}_{n-1}\partial_n^{\R}=0$. It is however an~interesting
question what properties of $\lt$ and $\star$ are really used. With relation to the~paper
\cite{NP-HomOper} we noticed that it is distributivity which makes
$\left(C^{\R}(X),\partial^{\R}\right)$ a~chain complex.
Furthermore, we observed that if $\star_1$ and $\star_2$ are self-distributive and distributive
one with respect to the~other then $\partial^{(a_1,a_2)}= a_1\partial^{\star_1}+a_2\partial^{\star_2}$
is a~differential for any scalars $a_1,a_2\in R$.
In a~full generality, we can take any distributive set $\{\star_1,\dots,\star_k\}$ and scalars
$a_1,\dots,a_k\in R$ to build a~chain complex $\left(C(X;R),\partial^{(a_1,\dots,a_k)}\right)$ with
\begin{equation}
	\partial^{(a_1,...,a_k)}=a_1\partial^{\star_1}+\dots+a_k\partial^{\star_k}.
\end{equation}
\noindent
It is called a~\emph{multi-term chain complex} for a~multishelf $X$
(see \cite{Prz-Demonstratio,P-S}).

As in classical theory, any morphism $f\colon X\to Y$ of multishelves
induces a~chain map of complexes $f_\sharp\colon C(X;R)\to C(Y;R)$,
which further descends to a~map between homology groups $f_*\colon H(X;R)\to H(Y;R)$.
The~basic examples are provided by an~inclusion of a~submultishelf or
a~projection $f(x)=x\star_r t$ of $X$ onto the~right orbit $Y=X\star_r t$
of $t$ with respect to of the~operations.

\subsection{Weak simplicial modules}

It is convenient to have terminology, usefulness of which will be visible in next sections and
which takes into account the~fact that in most homology theories a~boundary operation
$\partial_n\colon C_n \to C_{n-1}$ decomposes as an~alternating sum of \emph{face maps}
$d_i\colon C_n \to C_{n-1}$. Often we also have \emph{degeneracy maps} $s_j\colon C_n \to C_{n+1}$.
This section is motivated by~\cite{Lod}.

\begin{definition}\label{def:simpl-mod}
Let $R$ be a~ring. A \emph{simplicial module} $(C_n,d_i,s_j)$ is a~collection of
$R$-modules $C_n$ for $n\geq 0$, together with \emph{face maps} $d_i\colon C_n\to C_{n-1}$
and \emph{degeneracy maps} $s_j\colon C_n\to C_{n+1}$, $0\leq i,j \leq n$,
which satisfy the following properties:
\begin{enumerate}[label={\normalfont (SM\arabic{*})},ref=(SM\arabic{*})]
\itemindent 2em
\item\label{SM-dd} $d_id_j = d_{j-1}d_i$, for $i<j$
\item\label{SM-ss} $s_is_j = s_{j+1}s_i$, if $0\leq i \leq j \leq n$
\item\label{SM-ds} $d_is_j = \begin{cases}
							s_{j-1}d_i   & \textrm{ if }i<j \\
							s_{j}d_{i-1} &\textrm{ if }i>j+1
					 \end{cases}$
\item\label{SM-ds-id} $d_is_i=d_{i+1}s_i= \id_{C_n}$
\end{enumerate}

\noindent A~\emph{weak simplicial module} $(C_n,d_i,s_j)$ satisfies conditions
\ref{SM-dd}--\ref{SM-ds} and the~following weakened version of \ref{SM-ds-id}:
\begin{enumerate}[label={\normalfont (W4)},ref=(W4)]
\itemindent 2em
\item\label{WSM} $d_is_i=d_{i+1}s_i$
\end{enumerate}
\end{definition}

\noindent The~definition of a~weak simplicial module is new and motivated by
homology of distributive structures. We use the term \emph{weak} as the terms
\emph{pseudo-} and \emph{almost simplicial} modules are already in use.%
\footnote{
	According to \cite{Fra}, a~pseudo-simplicial module $(C_n,d_i,s_j)$ satisfies only
	conditions \ref{SM-dd}, \ref{SM-ds} and \ref{SM-ds-id} (see \cite{Ti-Vo,In})
	whereas an~\emph{almost simplicial module} satisfies \ref{SM-dd}--\ref{SM-ds-id}
	except $s_is_i = s_{i+1}s_i$.
}

For any~(weak) simplicial module $(C_n,d_i,s_j)$ there is a~natural chain complex
$(C,\partial)$ with the~differential being an~alternating sum of face maps
$\partial_n = \sum_{i=0}^n(-1)^i d_i$. Hence, only face maps are necessary
and the~collection $(C_n,d_i)$ satisfying \ref{SM-dd} is called a~\emph{presimplicial module}.
In particular, a~chain complex for a~multishelf is a~presimplicial module,
and for a~multispindle even more can be said.

\begin{proposition}\label{prop:multi-term-is-weak-simpl}
Let $(X,\{\star_1,\dots,\star_k\})$ be a~multispindle. For scalars $a_1,\dots,a_k\in R$
put $C_n:=R X^{n+1}$ and
\begin{align*}
d_i(x_0,\dots,x_n) &:= \sum_{r=1}^k a_r d_i^{\star_r}(x_0,\dots,x_n)\\
s_j(x_0,\dots,x_n) &:= (x_0,\dots,x_{j-1},x_j,x_j,x_{j+1}\dots,x_n)
\end{align*}
Then the~collection $(C_n,d_i,s_j)$ is a~weak simplicial module.
\end{proposition}

\noindent The~proof is given by a~direct calculation and is left as an~easy exercise.

\subsection{Subcomplex of degenerate elements}\label{sec:deg-subcompl}

For a~weak simplicial module $(C_n,d_i,s_j)$ we define the~\emph{degenerate submodule}
$C^D$ as a~submodule generated by images of degeneracy maps:
\begin{equation}
	C^D_n := \mathrm{span}\left\{ s_0(C_{n-1}),\dots,s_{n-1}(C_{n-1})\right\}.
\end{equation}
Conditions \ref{SM-ds} and \ref{WSM} are enough to provide that it is a~subcomplex
of $(C,\partial)$:
\begin{align*}
\partial_n s_p &= \Big(\sum_{i=0}^n (-1)^i d_i\Big)s_p = \sum_{i=0}^n(-1)^i(d_is_p) = \\
&= \sum_{i=0}^{p-1}(-1)^i(s_{p-1}d_i) + \sum_{i=p+2}^n(-1)^i(s_pd_{i-1})
\end{align*}
because $(-1)^p d_ps_p + (-1)^{p+1} d_{p+1}s_p = 0$.

It is a~classical result that if $(C_n,d_i,s_j)$ is a~simplicial module
then $C_n^D$ is acyclic. However, this does not hold for a~weak simplicial module
and we can have nontrivial degenerate homology $H^D(C)=H(C^D)$,
so that normalized homology $H^N(C)=H(C/C^D)$ is different from $H(C)$.
This plays an~important role in the theory of distributive homology.

In general, there is a~filtration on the~degenerate subcomplex
$\mathcal{F}^0\subset \mathcal{F}^1\subset ...\subset C^D$ defined as
\begin{equation}
	\mathcal{F}^p_{\!n} := \mathrm{span}\left\{ s_0(C_{n-1}),...,s_{p}(C_{n-1}) \right\}.
\end{equation}
The~fact that $\partial (\mathcal{F}^p) \subset \mathcal{F}^p$ follows from
the~relations between faces and degeneracies (again \ref{WSM} is enough).
In fact, the~calculation performed before shows that
$\partial(s_p(C_n))\subset \mathcal{F}^p_{\!n}$.
%$\partial s_p(C_n)\subset \mathrm{span}\left\{s_{p-1}(C_{n-1}), s_p(C_{n-1})\right\}
%\subset \mathcal{F}^p_{\!n}$.

\section{Decompositions of a~chain complex}\label{chpt:decomp}
We will show here how to decompose a~multi-term distributive chain complex for
a~multispindle into simpler pieces. Although the~results can be stated for any
coefficients, for simplicity we will restrict ourselves to the~ring of integers $\Z$.
From the point of view of computing homology
of distributive lattices, very important for us is the~decomposition of a~chain
complex $C(X)$ into a~chain complex of a~point $C(t)$, a~reduced (or pointed)
initially degenerate complex $\F(X,t)$ and an~initially normalized complex
$\CF(X,t) := C(X,t)/\!\F(X,t)$.
At the~end of the~section we show how to split $C(X)$ and $C(X,t)$
into degenerate and normalized parts.

\subsection{Submultishelves and subcomplexes}

Let $(X,\{\star_1,...,\star_k\})$ be a~multishelf. Recall that $A\subset X$ is
a~submultishelf if it is closed under all operations $\star_r$. In particular,
the~one element subset $\{t\}\subset X$ is a~submultishelf if and only if $t$
is an~idempotent for each operation (i.e. $t\star_r t=t$ for every $r$).
A~submultishelf $A\subset X$ induces a~short exact sequence of chain complexes
\begin{equation}
\label{eq:seq-subcomplex} 0 \to C(A) \to C(X) \to C(X,A) \to 0
\end{equation}
where $C_n(X,A) := C_n(X)/C_n(A)$. All groups are free, so it splits in every degree $n$.
In a~few interesting to us situations, it splits also as a~sequence of chain complexes.

\begin{proposition}\label{prop:compl-split}
Let $X$ be a~multishelf and $t\in X$.
\begin{enumerate}[label={\normalfont (\roman{*})},ref=(\roman{*})]
\item\label{prop-retract} If $A\subset X$ is given by a~multishelf retraction $r_A\colon X\to A$,
		the~sequence (\ref{eq:seq-subcomplex}) splits.
\item\label{prop-t-sub} If $A=\{t\}$ is a~submultishelf, then the~sequence
		(\ref{eq:seq-subcomplex}) splits.
\item\label{prop-ideal} Assume that $(x\star_r t)\star_r t = x\star_r t$ for some $\star_r$
		and every $x\in X$. Then the~sequence (\ref{eq:seq-subcomplex}) splits for
		$A=X\star_r t$ with a~retraction induced by $r_t(x):=x\star_r t$.
\end{enumerate}
\end{proposition}
\begin{proof}
The~retraction in~\ref{prop-retract} induces a~chain map
$(r_A)_\sharp\colon C(X)\to C(A)$ that is right inverse to the~inclusion
$C(A)\subset C(X)$. Parts~\ref{prop-t-sub} and \ref{prop-ideal} are special
cases of \ref{prop-retract}.
\end{proof}

\noindent Notice that \ref{prop-t-sub} is a~special case of~\ref{prop-ideal},
when one of the~operations is the~\projright, i.e. $\star_r = \rt$.
Furthermore, the~splitting for $C(X)\to C(X,A)$
is given by $\bar x\mapsto x-(r_A)_\sharp(x)$.
This will be used later for the~case $A=\{t\}$.

When an~exact sequence of chain complexes splits, so do homology groups.
In particular,
\begin{equation}
H_n(X) \cong H_n(X,t)\oplus H_n(t).
\end{equation}
The~complex $C(X,t)$ as well as homology $H(X,t)$ are called \emph{reduced}.
By 5-lemma, both are independent of $t$ as long as $\{t\}$ is a~submultishelf of $X$.

Proposition \ref{prop:compl-split} can be strengthened for pairs as follows.

\begin{proposition}\label{prop:compl-split-pairs}
Let $X$ be a~multishelf and $A\subset B\subset X$ be submultishelves.
Then there is a~short exact sequence of complexes
\begin{equation}\label{eq:seq-for-tripple}
0 \to C(B,A) \to C(X,A) \to C(X,B) \to 0
\end{equation}
natural with respect to maps of triples $f\colon (X,B,A)\to(X',B',A')$.
Moreover, if $B$ is a~multishelf retract of $X$, then
the~sequence (\ref{eq:seq-for-tripple}) splits.
\end{proposition}
\begin{proof}
Exactness follows from the~$3\times 3$-lemma applied to the~diagram
$$\xy
\morphism( 400,1400)/->/<0,-300>[0`\phantom{C(A)};]
\morphism(1000,1400)/->/<0,-300>[0`\phantom{C(B)};]
\morphism(1600,1400)/->/<0,-300>[0`\phantom{C(B,A)};]
\morphism(   0,1100)/->/<400,0>[0`\phantom{C(A)};]
\morphism( 400,1100)/->/<600,0>[C(A)`\phantom{C(B)};]
\morphism(1000,1100)/->/<600,0>[C(B)`\phantom{C(B,A)};]
\morphism(1600,1100)/->/<500,0>[C(B,A)`0;]
\morphism( 400,1100)/=/< 0,-400>[\phantom{C(A)}`\phantom{C(A)};]
\morphism(1000,1100)/->/<0,-400>[\phantom{C(B)}`\phantom{C(X)};]
\morphism(1600,1100)/->/<0,-400>[\phantom{C(B,A)}`\phantom{C(X,A)};]
\morphism(   0,700)/->/<400,0>[0`\phantom{C(A)};]
\morphism( 400,700)/->/<600,0>[C(A)`\phantom{C(X)};]
\morphism(1000,700)/->/<600,0>[C(X)`\phantom{C(X,A)};]
\morphism(1600,700)/->/<500,0>[C(X,A)`0;]
\morphism( 400,700)/->/<0,-400>[\phantom{C(A)}`\phantom{0};]
\morphism(1000,700)/->/<0,-400>[\phantom{C(X)}`\phantom{C(X,B)};]
\morphism(1600,700)/->/<0,-400>[\phantom{C(X,A)}`\phantom{C(X,B)};]
\morphism( 400,300)/->/<600,0>[0`\phantom{C(X,B)};]
\morphism(1000,300)/=/< 600,0>[C(X,B)`\phantom{C(X,B)};]
\morphism(1600,300)/->/<500,0>[C(X,B)`0;]
\morphism(1000,300)/->/<0,-300>[\phantom{C(X,B)}`0;]
\morphism(1600,300)/->/<0,-300>[\phantom{C(X,B)}`0;]
\endxy$$
and naturality is obvious from definition.

For the~second part, let $r\colon X\to B$ be a~multishelf retraction.
Then $r|_A = \id_A$ and there is a~chain map $r_\sharp\colon C(X,A)\to C(B,A)$
that splits the~sequence.
\end{proof}

Sometimes it is convenient to allow $A$ to be empty,
in which case $C(X,\emptyset) = C(X)$.

\subsection{The~initial degenerate subcomplex}\label{sec:init-deg}
Assume for now that $X$ is a~multispindle and $A\subset X$ is
a~submultispindle. Proposition \ref{prop:multi-term-is-weak-simpl}
and the~constructions from section~\ref{sec:deg-subcompl} can be easily
extended for a~pair $(X,A)$, so that we have a~subcomplex $\F(X,A)$.
It is straightforward to check that $\F$ is natural with respect
to maps between pairs of multispindles.

Consider now the~following two maps:
\begin{align*}
	\sigma\colon& C_n\to C_{n+1}, & \sigma(x_0,\dots,x_n) &= (-1)^{n+1} s_0(x_0,\dots,x_n),\\
	\pi\colon& C_n\to C_{n-1}, & \pi(x_0,\dots,x_n) &= (-1)^n d_0^{\lt}(x_0,\dots,x_n),
\end{align*}
and let $\Sigma=\sum_{i=1}^k a_i$ be the~sum of scalars defining $\partial^{(a_1,...,a_k)}$.

\begin{lemma}\label{lem:s0-htpy}
The~map $s_0\colon C_n(X,A)\to C_{n+1}(X,A)$ is a~chain homotopy between 
$\Sigma\cdots\sigma\pi$ and the zero map.
In particular, a~composition $\sigma\pi\colon C(X,A)\to C(X,A)$ is a~chain map.
\end{lemma}
\begin{proof}
We use the fact that $d_0s_0=d_1s_0= \Sigma\cdot\id_{C_n}$ and that 
$(C_n(X,A),d_i,s_j)$ is a a weak simplicial module. In particular,
$d_is_0=s_0d_{i-1}$ for $i>1$ and we have:
\begin{align*}
\partial^{(a_1,...,a_k)}s_0 &+ s_0\partial^{(a_1,...,a_k)}
	= \sum_{i=0}^{n+1}(-1)^{i}d_is_0 + \sum_{i=0}^n(-1)^{i}s_0d_i = \\
&= (d_0s_0- d_1s_0) + \sum_{i=2}^{n+1}(-1)^{i}d_is_0 + \sum_{i=0}^n(-1)^{i}s_0d_i =\\
&= \sum_{i=2}^{n+1}(-1)^{i}s_0d_{i-1} + \sum_{i=0}^n(-1)^{i}s_0d_i = \\
&= \sum_{i=1}^{n+1}(-1)^{i+1}s_0d_{i} + \sum_{i=0}^n(-1)^{i}s_0d_i= s_0d_0
= (\sum_{i=1}^k a_i) \sigma\pi.
\end{align*}
Therefore, $\Sigma\cdot\sigma\pi$ is a chain map and as all groups $C_n(X,A)$
are free, the~composition $\sigma\pi$ is a chain map.
\end{proof}

\begin{corollary}\label{cor:split-F0}\ 

\begin{enumerate}[label={\normalfont (\roman{*})},ref=(\roman{*})]
\item\label{cor-Sigma-anih} $\Sigma$ annihilates $H_n(\F(X,A))$.
\item\label{cor-Sigma-0} If $\Sigma=0$ then both $\sigma$ and $\pi$ are chain maps.
\item\label{cor-C-F0-Sigma-isom} The~maps $\sigma$ and $\pi$ induce isomorphisms
of chain complexes
$$
C_\bullet(X,A;\Z_\Sigma)\cong \F_{\bullet+1}(X,A;\Z_\Sigma).
$$
In particular, $H_n(X,A;\Z_\Sigma)\cong H_{n+1}(\F(X,A);\Z_\Sigma)$.
\end{enumerate}
\end{corollary}
\begin{proof}
Part~\ref{cor-Sigma-anih} follows directly from Lemma~\ref{lem:s0-htpy},
because $\sigma\pi$ is an~identity on $\F(X,A)$.
If $\Sigma=0$, then Lemma~\ref{lem:s0-htpy} implies $\sigma$ is a~chain map
and similarly for $\pi$, what gives \ref{cor-Sigma-0}.
Finally, \ref{cor-C-F0-Sigma-isom} is a~consequence of the~above.
\end{proof}

For us the~most important consequence of Lemma~\ref{lem:s0-htpy}
is the~following fact.

\begin{proposition}\label{prop:split-F0}
There is a~natural short exact sequence of chain complexes
\begin{equation}
0 \to \F(X,A) \to^i C(X,A) \to^p \CF(X,A) \to 0
\end{equation}
which splits. Namely, there is a~retraction $\sigma\pi\colon C(X,A) \to \F(X,A)$
and a~splitting map $(\id-\sigma\pi)\colon \CF(X,A) \to C(X,A)$, and both are natural
with respect to maps of pairs of multispindles.
\end{proposition}
\begin{proof}
We have checked that $\sigma\pi$ is a chain map (Lemma~\ref{lem:s0-htpy}).
Furthermore, $\sigma\pi\circ i=\id$ implies $\sigma\pi$ is a~retraction and the~sequence splits.
For naturality, take a~map of multispindles $f\colon (X,A)\to (Y,B)$ and notice that
$\sigma\circ f_\sharp = f_\sharp\circ\sigma$ and similarly for $\pi$.
\end{proof}

This result when combined with Propositions~\ref{prop:compl-split}
and~\ref{prop:compl-split-pairs} provides the~following fact.

\begin{proposition}\label{prop:F-C-CF-retract-agree}
Let $X$ be a~multispindle and $A\subset B\subset X$ submultispindles. Then in the~following
diagram all rows and columns are exact.
$$\xy
\morphism( 500,1400)/->/<0,-300>[0`\phantom{\F(B,A)};]
\morphism(1200,1400)/->/<0,-300>[0`\phantom{C(B,A)};]
\morphism(1900,1400)/->/<0,-300>[0`\phantom{\CF(B,A)};]
\morphism(   0,1100)/->/<500,0>[0`\phantom{\F(B,A)};]
\morphism( 500,1100)/->/<700,0>[\F(B,A)`\phantom{C(B,A)};]
\morphism(1200,1100)/->/<700,0>[C(B,A)`\phantom{\CF(B,A)};]
\morphism(1900,1100)/->/<500,0>[\CF(B,A)`0;]
\morphism( 500,1100)/->/< 0,-400>[\phantom{\F(B,A)}`\phantom{\F(X,A)};]
\morphism(1200,1100)/->/<0,-400>[\phantom{C(B,A)}`\phantom{C(X,A)};]
\morphism(1900,1100)/->/<0,-400>[\phantom{\CF(B,A)}`\phantom{\CF(X,A)};]
\morphism(   0,700)/->/<500,0>[0`\phantom{\F(X,A)};]
\morphism( 500,700)/->/<700,0>[\F(X,A)`\phantom{C(X,A)};]
\morphism(1200,700)/->/<700,0>[C(X,A)`\phantom{\CF(X,A)};]
\morphism(1900,700)/->/<500,0>[\CF(X,A)`0;]
\morphism( 500,700)/->/<0,-400>[\phantom{\F(X,A)}`\phantom{\F(X,B)};]
\morphism(1200,700)/->/<0,-400>[\phantom{C(X,A)}`\phantom{C(X,B)};]
\morphism(1900,700)/->/<0,-400>[\phantom{\CF(X,A)}`\phantom{\CF(X,B)};]
\morphism(   0,300)/->/<500,0>[0`\phantom{\F(X,B)};]
\morphism( 500,300)/->/<700,0>[\F(X,B)`\phantom{C(X,B)};]
\morphism(1200,300)/->/<700,0>[C(X,B)`\phantom{\CF(X,B)};]
\morphism(1900,300)/->/<500,0>[\CF(X,B)`0;]
\morphism( 500,300)/->/<0,-300>[\phantom{\F(X,B)}`0;]
\morphism(1200,300)/->/<0,-300>[\phantom{C(X,B)}`0;]
\morphism(1900,300)/->/<0,-300>[\phantom{\CF(X,B)}`0;]
\endxy$$
Furthermore, if $B$ is a~retract of $X$, then all columns split and the~split maps
commute with the~horizontal arrows.
\end{proposition}

The~main results are summarized in the~corollary below.
They are used later in computation of four-term homology of distributive lattices.

\begin{corollary}\label{cor:Ct-F0-split}
If $X$ is a~multispindle, then the~chain complex $C(X)$ decomposes as
\begin{equation}
C(X) \cong C(t) \oplus \F(X,t) \oplus \CF(X,t).
\end{equation}
In particular,
\begin{equation}
\label{eq:H-decomp} H_n(X) \cong H_n(t) \oplus H_n(\F(X,t)) \oplus H_n(\CF(X,t)).
\end{equation}
\end{corollary}

\subsection{Degenerate and normalized subcomplexes}

For a~quandle $(X,\star)$ and its chain complex $(C^{\R}(X),\partial^{\R})$,
Carter, Kamada and Saito \cite{Car,CJKLS} considered the~degenerate subcomplex
and the~quotient, called a~quandle complex.
Litherland and Nelson \cite{L-N} proved that this complex splits and
their result extends to any~multispindle. Our proof follows that given in \cite{N-P-2}.

\begin{theorem}\label{thm:Norm-split}
Let $X$ be a~multispindle. Then the~short exact sequence of chain complexes
\begin{equation}\label{eq:deg-norm-seq}
0 \to C^D(X) \to^i C(X) \to^\beta C^N(X) \to 0
\end{equation}
splits with a~split map $\alpha\colon C^N(X) \to C(X)$ given by the~formula%
\footnote{
	We use here a~multilinear convention as in \cite{N-P-2}, e.g.
	\begin{align*}
	\alpha(x_0,x_1,x_2)&=(x_0,x_1-x_0,x_2-x_1)=
			(x_0,x_1,x_2)-(x_0,x_0,x_2) - (x_0,x_1,x_1)+ (x_0,x_0,x_1).
	\end{align*}
}
\begin{equation}\label{eq:alpha}
\alpha(x_0,x_1,x_2,...,x_n) = (x_0,x_1-x_0,x_2-x_1,...,x_n-x_{n-1}).
\end{equation}
In particular, $H_n(X) \cong H_n^D(X) \oplus H_n^N(X)$.
\end{theorem}
\begin{proof}
Observe that the~map $\alpha$ is well defined as
$$
\alpha(s_i(x_0,...,x_{n-1}))= (x_0,...,x_i-x_i,...,x_{n-1}) = 0,
$$
so that $\alpha (C_n^D)=0$. We have also $\beta\alpha = \id_{C^N}$,
since $(\alpha\beta - \id)(C_n) \subset C^D_n$ and $\beta (C_n^D)=0$.
This shows that $\alpha$ splits $\{C_n\}$ as a graded group.
It remains to check that $\alpha$ is a chain map and for this
it suffices to prove that $\partial^{\star_r}\alpha = \alpha\partial^{\star_r}$
for any $r$. This follows from Lemma \ref{lem:alpha-chain} below.
\end{proof}

\begin{lemma}\label{lem:alpha-chain}\ 
Let $(C_n,d_i,s_j)$ be a~weak simplicial complex associated to a~spindle
$(X,\star)$. Then $\alpha\partial_n^\star = \partial_n^\star\alpha$ for every $n\geqslant 0$,
where $\alpha$ is defined by formula~(\ref{eq:alpha}).
\end{lemma}
\begin{proof}
Denote halves of $d_i\alpha$ as follows:
\begin{align*}
A_i &:= d_i(x_0,x_1\!-\!x_0,\dots,x_{i-1}\!-\!x_{i-2},\phantom{{}_{-1}}x_i,x_{i+1}\!-\!x_i,\dots,x_n\!-\!x_{n-1}),\\
B_i &:= d_i(x_0,x_1\!-\!x_0,\dots,x_{i-1}\!-\!x_{i-2},x_{i-1},x_{i+1}\!-\!x_i,\dots,x_n\!-\!x_{n-1}),
\end{align*}
so that $d_i\alpha(x_0,\dots,x_n) = A_i-B_i$. Direct calculation gives
\begin{align*}
B_i &= \begin{cases}
	0, & i=0,\\
	\alpha d_{i-1}(x_0,\dots,x_n)-A_{i-1}, & i >0,
\end{cases} \\
A_n &= \alpha d_n(x_0,\dots,x_n),
\end{align*}
and by summing everything up we have $\partial^\star\alpha = \alpha\partial^\star$.
\end{proof}

\section{Computation of multi-term homology}\label{chpt:hom}
Now we are going to compute homology of a~complex defined in the~section~\ref{chpt:complex}.
The~main idea is to use decompositions described previously to compute
homology groups for the~two-element Boolean algebra $B_1$. Then we will show that
homology of any finite distributive lattice can be reduced to those simple pieces.

\subsection{General statements}
Let $(X,\{\star_1,\dots,\star_k\})$ be a~multispindle and choose scalars $a_1,\dots,a_k$.
Recall that $\Sigma$ is the~sum of all $a_i$'s.
By Corollary~\ref{cor:Ct-F0-split} the~multiterm chain complex $C(X)$ decomposes as
\begin{equation}
\label{eq:C-decomp} C(X) = C(\{t\})\oplus \F(X, t)\oplus \CF(X, t).
\end{equation}

\begin{proposition}\label{prop:one-point-hom}
Let $t\in X$. If $\Sigma = 0$, then $H_n(\{t\}) = \Z$ for every $n \geqslant 0$.
Otherwise,
$$
H_n(\{t\})=\begin{cases}
	\Z, & n = 0, \\
	0, & n > 0 \textrm{ and is even},\\
	\Z_{\Sigma}, & n \textrm{ is odd}.
\end{cases}
$$
\end{proposition}

\noindent Form this proposition one can immediately deduce augmented homology groups
--- the only difference is in degree 0, in which $\tilde H_0(\{t\}) = 0$.

\begin{proof}
Because all operations are idempotent, we have
$$
d_i(t,\dots,t) = \sum_{r=0}^k a_r d_i^{\star_r}(t,\dots,t) = \Sigma\cdot (t,\dots,t)
$$
what gives
$$
\partial_n(t,\dots,t) = \begin{cases}
	\Sigma\cdot (t,\dots,t), & n \textrm{ is even}\\
	0, & n \textrm{ is odd}
\end{cases}
$$
and the~proposition follows.
\end{proof}

Homology of the~second factor in~(\ref{eq:C-decomp}) can be computed recursively.
Take a~map $\sigma\colon C_\bullet(X,t)\to C_{\bullet+1}(X,t)$ defined
in section~\ref{sec:init-deg}. By Corollary~\ref{cor:split-F0} it induces an~isomorphism
$H_{n+1}(\F(X,t))\cong H_n(X,t)$
if $\Sigma = 0$. Otherwise, we have to work a~bit harder.

\begin{proposition}\label{prop:F0-hom}
Let $X$ be a~finite multispindle with $t\in X$. Assume $\Sigma = 0$ or the~free
part of $H_n(X,t)$ is trivial for every $n\geqslant 0$. Then
\begin{equation}
	H_{n+1}(\F(X,t))\cong H_n(X,t)\otimes\Z_\Sigma
\end{equation}
for every $n\geqslant 0$. On the~other hand, if $\Sigma\neq 0$ and
$H_n(\CF(X,t))$ is a~free group for every $n\geqslant 0$, then
\begin{equation}
	H_{n+1}(\F(X,t))\cong H_{n-1}(\F(X,t))\ \oplus\ H_n(\CF(X,t))\otimes\Z_\Sigma
\end{equation}
for every $n\geqslant 0$.
\end{proposition}
\begin{proof}
The~case $\Sigma=0$ is discussed above. For $\Sigma\neq 0$
the~Universal Coefficient Theorem implies that
\begin{align*}
H_n(X,t;\Z_\Sigma)&\cong Tor(H_{n-1}(X,t),\Z_\Sigma) \oplus H_n(X,t)\otimes\Z_\Sigma,\\
H_{n+1}\F(X,t;\mathbb{Z}_\Sigma)&\cong Tor(H_n\F(X,t),\Z_\Sigma)\oplus H_{n+1}\F(X,t)\otimes\Z_\Sigma.
\end{align*}
The groups on the~left-hand side are isomorphic and
Corollary~\ref{cor:split-F0} implies that $\Sigma$ annihilates
each $H_n\F(X,t)$, so that
$$
Tor(H_n\F(X,t),\Z_\Sigma)\cong H_n\F(X,t)\otimes\Z_\Sigma.
$$
If homology groups $H_n(X,t)$ have no free part, then
$$
Tor(H_n(X,t),\Z_\Sigma)\cong H_n(X,t)\otimes\Z_\Sigma
$$
and in case $H_n(\CF(X,t))$ is free, we have
$$
Tor(H_n(X,t),\Z_\Sigma)\cong Tor(H_n\F(X,t),\Z_\Sigma)\cong H_n\F(X,t).
$$
Putting everything together, in the~first case we obtain an isomorphism
$$
H_{n-1}(X,t)\otimes\Z_\Sigma \oplus H_n(X,t)\otimes\Z_\Sigma
\cong
H_n\F(X,t) \oplus H_{n+1}\F(X,t)
$$
and a~simple inductive argument proves the~thesis, because
all groups are finite (hence, finitely generated).
For the~other case
$$
H_{n-1}\F(X,t) \oplus H_n(X,t)\otimes\Z_\Sigma
\cong
H_n\F(X,t) \oplus H_{n+1}\F(X,t).
$$
and since $H_n(X,t)\cong H_n(\CF(X,t))\oplus H_n(\F(X,t))$
a~simple cancellation is all we need (again, all groups are finitely generated).
\end{proof}

We ends this section with one more remark for general complexes
that will be used a~few times in the~next section.

\begin{lemma}\label{lem:diff-qdiff}
Let $(C,\partial)$ be a~chain complex of finitely generated free abelian groups
with $n$-th homology group equal to
\begin{equation}
H_n(C,\partial)=\Z^k\oplus\Z_{m_1}\oplus\dots\oplus\Z_{m_r}.
\end{equation}
Then for any positive integer $q$ the $n$-th homology group of $(C,q\partial)$ is equal to
\begin{equation}\label{eq:H-gdiff}
H_n(C,q\partial)=\Z^k\oplus\Z_{qm_1}\oplus\dots\oplus\Z_{qm_r}\oplus\Z_q^{t_n-k-r}
\end{equation}
where the~sequence $t_n$ is given by a~recursive formula
\begin{equation}\label{eq:H-qdiff-1's}
t_{n+1} + t_n = \rk C_{n+1} + \rk H_n(C,\partial)
\end{equation}
with $t_0 = \rk C_0$.
\end{lemma}

\noindent The~factors $\Z_q$ in (\ref{eq:H-gdiff}) can be seen as arising from trivial
factors $\Z_1 = 0$.

\begin{proof}
Because all chain groups are free, $q\partial$ has the~same kernel as $\partial$.
On the~other hand, the~image is ``multiplied by $q$''. Formally, denote by $M$
a~matrix representing
$
\partial_{n+1}\colon C_{n+1}\to \ker\partial_n
$
and by $M'$ the~analogous matrix for $q\partial_{n+1}$.
Then clearly $M' = q\cdot M$ and when we diagonalize both matrices simultaneously,
entries on the~diagonal of $M'$ are equal to the~entries of $M$ multiplied by $q$.
Since $M$ is a~presentation matrix for $H_n(C,\partial)$
and $M'$ for $H_n(C,q\partial)$, it remains to find the~number of $1$'s in $M$.
This is equal to $\rk(\ker\partial_n)-k-r$ and from linear algebra we know that
$$
\rk(\ker\partial_n) +\rk(\ker\partial_{n-1}) = \rk C_n +\rk H_{n-1}
$$
what proves the~formula (\ref{eq:H-qdiff-1's}) and gives the~thesis.
\end{proof}

\subsection{Computation for the~Boolean algebra \texorpdfstring{$B_1$}{B1}}\label{sec:hom-of-B1}
Now we go back to distributive lattices. Here we present computation
for the~two-element Boolean algebra, which is the~basic piece. By the~previous part
it remains to show that reduced homology groups are finite and
to compute homology for the~third piece of (\ref{eq:C-decomp}).

For generic scalars $a,b,c$ and $d$, homology of any lattice
are purely torsion due to the~following.

\begin{proposition}\label{prop:ac-ab-ann-H}
Let $(C,\partial)$ be a~four-term chain complex associated to
a~distributive lattice $L$ for scalars $(a,b,c,d)$. Then the~reduced
homology groups $H(L,t)$ are annihilated by $\gcd(a+b,a+c)$.
\end{proposition}
\begin{proof}
Assume first that $L$ is complete and consider the~following homotopies:
\begin{align*}
h^\bot(x_0,\dots,x_n) &:= (-1)^{n+1}(x_0,\dots,x_n,\bot),\\
h^\top(x_0,\dots,x_n) &:= (-1)^{n+1}(x_0,\dots,x_n,\top).
\end{align*}
For a~fixed $t\in L$, according to Proposition~\ref{prop:compl-split}\ref{prop-t-sub},
$C(L,t)$ as a~subcomplex of $C(L)$ is generated by elements $(x_0,\dots,x_n)-(t,\dots,t)$.
Therefore, when restricted to $C(L,t)$:
\begin{align*}
\partial h^\bot + h^\bot\partial &= (a+b)\id, \\
\partial h^\top + h^\top\partial &= (a+c)\id.
\end{align*}
Hence, both $(a+b)$ and $(a+c)$ annihilates $H(L,t)$.

For a~general case, notice that any chain $w\in C(L,t)$
is finitely supported, i.e. it can be expressed using only
finitely many elements $x_0,\dots,x_k\in L$. Assuming $w$ represents
a~class $[w]\in H(L,t)$, replace $\top$ with $x_0\lor\cdots\lor x_k$
and $\bot$ with $x_0\land\cdots\land x_k$ in the~above
to see that both $(a+b)$ and $(a+c)$ annihilates $[w]$.
\end{proof}

Since now we focus on the~Boolean algebra $B_1$. For the~rest of this
part $(C,\partial)$ denotes the~four-term chain complex associated to $B_1$
for scalars $(a,b,c,d)$.
Notice that $\partial^{\rt}$ is zero on $\CF(B_1,\bot)$ and is of no concern for us.
We will prove first that $\partial$ is divisible by $\gcd(a+b,a+c)$
and that after dividing by this number the~complex is acyclic.

\begin{lemma}\label{lem:C-F0-red}
Let $(C,\partial)$ be the~four-term chain complex for $(B_1,\bot)$ with scalars
$(a,b,c,d)$ and put $g=\gcd(a+b,a+c)$. Then if $g=0$, the~differential $\partial$
vanishes on $\CF(B_1,\bot)$. Otherwise,
\begin{enumerate}
\item $g$ divides the~differential $\partial$ restricted to $\CF(B_1,\bot)$,
\item the~complex $(\CF(B_1,\bot),\partial')$ is acyclic, where $g\partial'=\partial$.
\end{enumerate}
\end{lemma}
\begin{proof}
Notice first that $d_i^{\lor}(x_0,\dots,x_n) = d_i^{\lt}(x_0,\dots,x_n)$ if $x_i=\bot$
and $d_i^{\land}(x_0,\dots,x_n) = d_i^{\lt}(x_0,\dots,x_n)$ if $x_i=\top$.
This observation gives the~following equalities:
\begin{align*}
\partial(x_0,\top,x_2,\dots,x_n) &=
		  (a+c)\left[(\top,x_2,\dots,x_n)-(x_0,x_2,\dots,x_n)\right]\\
		&+(a+c)\sum_{i>1,x_i=\top}(-1)^i d_i^{\lt}(x_0,\top,x_2,\dots,x_n) \\
		&+(a+b)\sum_{i>1,x_i=\bot}(-1)^i d_i^{\lt}(x_0,\top,x_2,\dots,x_n),
\end{align*}
\begin{align*}
\partial(x_0,\bot,x_2,\dots,x_n) &=
		  (a+b)\left[(\bot,x_2,\dots,x_n)-(x_0,x_2,\dots,x_n)\right]\\
		&+(a+c)\sum_{i>1,x_i=\top}(-1)^i d_i^{\lt}(x_0,\bot,x_2,\dots,x_n) \\
		&+(a+b)\sum_{i>1,x_i=\bot}(-1)^i d_i^{\lt}(x_0,\bot,x_2,\dots,x_n).
\end{align*}
This shows the~first part of the~lemma.
The~second one follows from Proposition~\ref{prop:ac-ab-ann-H},
since for a~nonzero $g$ and $\partial' = \partial/g$ we have
\begin{align*}
	\partial' h^\bot + h^\bot\partial' &= \left(\frac{a+b}{g}\right)\id, \\
	\partial' h^\top + h^\bot\partial' &= \left(\frac{a+c}{g}\right)\id.
\end{align*}
But $g$ was chosen so that the~numbers on the~right-hand side are co-prime.
Hence, homology groups of $(\CF(B_1,\bot),\partial')$ are trivial.
\end{proof}

Therefore, if $(a+b)$ and $(a+c)$ are co-prime, the~third factor in~(\ref{eq:C-decomp})
adds nothing to homology. The~general case follows from Lemma~\ref{lem:diff-qdiff}.

\begin{proposition}\label{prop:F0-C-hom}
Let $(a,b,c,d)$ be any scalars. Then for $b=c=-a$
\begin{equation}\label{prop-F0-C-hom-triv}
H_n(\CF(B_1,\bot))=\Z^{2^n}
\end{equation}
and otherwise
\begin{equation}\label{prop-F0-C-hom-torsion}
H_n(\CF(B_1,\bot))=\Z_{\gcd(a+b,a+c)}^{r_n},
\end{equation}
where $r_n = \begin{cases}
	\frac{1}{3}(2^{n+1}+1),&n\textrm{ is even},\\
	\frac{1}{3}(2^{n+1}-1),&n\textrm{ is odd}.
\end{cases}$
\end{proposition}
\begin{proof}
Lemma~\ref{lem:C-F0-red} implies the~differential is trivial when both
$(a+b)$ and $(a+c)$ are zero, what gives~(\ref{prop-F0-C-hom-triv}).
In the~other case, due to Lemma~\ref{lem:diff-qdiff},
$H_n$ consists only of summands $\Z_{\gcd(a+b,a+c)}$ in the~power
equal to $r_n=\rk\ker\partial_n$, where
$$
r_n + r_{n-1} = \rk\CF\!\!_n(B_1,\bot) = 2^n,
$$
since for $n>0$ the~item $x_0$ in $(x_0,\dots,x_n)$ is determined by $x_1$ and
all other items can be chosen freely and when $n=0$ the chain group is generated
by a~single element $\top$. Finally,
$$
r_n = 2^n-r_{n-1} = 2^n-2^{n-1}+\ldots+(-1)^n = \frac{1}{3}\left(2^{n+1}+(-1)^n\right),
$$
what proves the~formula for $r_n$.
\end{proof}

\begin{remark}
The~sequence $r_n$ satisfies a~Chebyshev-type equation
\begin{equation}
r_{n+2} = r_{n+1}+2r_n
\end{equation}
with initial values $r_0=r_1=1$. There is a~similar exponential growth for
homology of dihedral quandles \cite{N-P-2} which was explained by existence
of homological operations \cite{Cla,N-P-3}. This suggested there should be a~similar
story for homology of a~distributive lattice. However, our search for such
operations was not successful.
\end{remark}

Notice that if $\Sigma\neq 0$, then either homology groups $H_n(X,t)$ are finite
(if $a+b\neq 0$ or $a+c\neq 0$) or $H_n(\CF(X,t))$ are free.
Therefore we can apply Proposition~\ref{prop:F0-hom} to compute
homology groups of $\F(B_1,\bot)$. Putting everything together, we obtain
the~following result.

\begin{proposition}\label{prop:B1-hom}
Let $(a,b,c,d)$ be any integer numbers and put $g=\gcd(a+b,a+c)$ and $\Sigma=a+b+c+d$.
Then:
$$
H_n(B_1,\bot)=\begin{cases}
	\Z^{2^{n+1}-1},												& \Sigma=0,g=0,\\
	\Z_g^{2r_n-p(n+1)},											& \Sigma=0,g\neq 0,\\
	\Z^{2^n}\oplus\Z_\Sigma^{r_n-p(n+1)},					& \Sigma\neq 0,g=0,\\
	\Z_g^{r_n}\oplus\Z_{\gcd(g,\Sigma)}^{r_n-p(n+1)},	& \Sigma\neq 0,g\neq 0
\end{cases}
$$
where $p(n)$ is the~parity of $n$, and the~sequence $r_n$ is defined as in
Proposition~\ref{prop:F0-C-hom}.
\end{proposition}
\begin{proof}
It remains to find, what $H_n(\F(B_1,\bot))$ adds to homology.
This can be derived directly from Proposition~\ref{prop:F0-hom}.
If $g=0$ and $\Sigma=0$, then we have to add another $s_n$ copies
of $\Z$, where
$$
s_n = 2^{n-1} + s_{n-1} = 2^{n-1} + 2^{n-2}+\ldots+ 1 = 2^n-1.
$$
In case $\Sigma\neq 0$ the~amount $s_n$ of additional copies of $\Z_\Sigma$ is given by
\begin{align*}
s_n
	&= 2^{n-1} + s_{n-2} = 2^{n-1} + 2^{n-3} + s_{n-4} = \\
	&= \begin{cases}
			\frac{2}{3}\left(4^{n/2}-1\right) = \frac{1}{3}\left(2^{n+1}+1\right)-1,& n\textrm{ is even},\\
			\frac{1}{3}\left(4^{(n+1)/2}-1\right) = \frac{1}{3}\left(2^{n+1}-1\right),& n\textrm{ is odd},\\
		\end{cases}
\end{align*}
so that $s_n = r_n-p(n+1)$.
Finally, when $g\neq 0$, the~amount $s_n$ of copies of $\Z_{\gcd(g,\Sigma)}$ is given by
the~other recursion formula:
\begin{align*}
3s_n
	&= 3r_{n-1}+3s_{n-1} = 2^n+(-1)^{n-1}+3s_{n-1} = \\
	&= 2^n + 2^{n-1} + \ldots + 2 + (-1)^{n-1} + \ldots + (-1)^0 =\\
	&= 2^{n+1}-2+p(n) = 2^{n+1}+(-1)^n-3p(n+1)
\end{align*}
and again $s_n = r_n-p(n+1)$.
\end{proof}

\begin{remark}\label{rmk:B1-F-direct}
Homology of $\F(B_1,\bot)$ can be computed directly
in a~similar way as we did for the~quotient $\CF(B_1,\bot)$.
Indeed, the~proof of Lemma~\ref{lem:C-F0-red} shows also that $\partial$
when restricted to $\F(B_1,\bot)$ is zero if $g=0$ or divisible by
$\gcd(g,\Sigma)$ in the~other case. Now use Proposition~\ref{prop:ac-ab-ann-H}
and the~fact that $\Sigma$ annihilates each $H_n(\F(B_1,\bot))$
to show that the~differential $\partial/\!\gcd(g,\Sigma)$ is acyclic.
\end{remark}

\subsection{Computation for any distributive lattice}
Now we are ready to compute homology groups for any distributive lattice.
We will use the~following splitting of homology groups.

\begin{lemma}\label{lem:H-split}
Let $L$ be a~distributive lattice and $t\in L$.
Pick any element $y\in L$ and assume $\gcd(a,b,c)=1$. Then
\begin{equation}
H_n(L, t) = H_n(\uparrow\!\! y, t\lor y) \oplus H_n(\downarrow\!\! y, t\land y).
\end{equation}
In particular, if $L$ is finite, then
\begin{equation}\label{eq:H-for-abc=1}
H_n(L,t) \cong H_n(B_1,\bot)^{\oplus J}.
\end{equation}
where $J$ is the~number of non-minimal irreducible elements in $L$.
\end{lemma}
\begin{proof}
Considering a~homotopy
$
h^y(x_0,\dots,x_n) := (x_0,\dots,x_n,y)
$
we have for any generator $\underline x := (x_0,\dots,x_n)-(t,\dots,t)\in C_n(L,t)$:
$$
\partial h^y(\underline x) + h^y(\partial\underline x) 
	= a\underline x + b(\underline x\lor y) + c(\underline x\land y).
$$
Therefore, for any class $\alpha\in H_n(L,t)$ we obtain
$$
a\alpha\in H_n(\uparrow\!\! y,t\lor y)+H_n(\downarrow\!\! y,t\land y)
$$
and since $(a+b)\alpha = (a+c)\alpha = 0$ by Proposition~\ref{prop:ac-ab-ann-H}
and $\gcd(a,b,c)=1$ it shows
$$
H_n(L,t) = H_n(\uparrow\!\! y,t\lor y)+H_n(\downarrow\!\! y,t\land y).
$$
Finally, the~above is a~direct sum, as both
subcomplexes $C(\uparrow\!\! y, t\lor y)$ and $C(\downarrow\!\! y, t\land y)$ are
direct summands of $C(L,t)$ due to Proposition~\ref{prop:F-C-CF-retract-agree}
and they are disjoint.

For the~last remark pick a~maximal chain $\mathcal{L}\subset L$ and
use the~above decomposition for every $y\in\mathcal{L}$.
The~thesis follows, because due to Theorem~\ref{thm:irred-chain}
for a~finite distributive lattice the~length of $\mathcal{L}$ equals
the~amount of non-minimal irreducible elements in $L$.
\end{proof}

Notice, that Proposition~\ref{prop:F-C-CF-retract-agree} makes
the~lemma above hold for $H_n(\F(L,t))$ and $H_n(\CF(L,t))$ as well.
Hence, we have all tools we need to compute homology for any~finite
distributive lattice.

\begin{theorem}\label{thm:hom-all}
Let $L$ be a~finite distributive lattice, $t\in L$ its element and
denote by $J$ the~amount of non-minimal irreducible elements in $L$.
Consider a~four-term chain complex for $L$ with scalars $(a,b,c,d)$
and let $\Sigma=a+b+c+d$, $g=\gcd(a+b,a+c)$, $g_3=\gcd(a,b,c)$
and $g_4=\gcd(a,b,c,d)$. Then,
$$
H_n(\CF(L,t)) = \begin{cases}
	\,\Z^{|L|^n(|L|-1)},															& g=0,a=0,\\
	\,\Z^{J2^n}     \oplus \Z_{g_3}^{(|L|-1)r_{n,|L|}-Jr_{n,2}},	& g=0,a\neq 0,\\
	\Z_g^{Jr_{n,2}}\!\oplus \Z_{g_3}^{(|L|-1)r_{n,|L|}-Jr_{n,2}}, 	& g\neq 0,
\end{cases}
$$
and
$$
H_n(\F(L,t)) = \begin{cases}
	\Z^{|L|^n-1},															& \Sigma=0,g=0,a=0,\\
	\Z_{g_4}^{r_{n,|L|}-p(n+1)},										& \Sigma\neq 0,g=0,a=0,\\
	\Z^{J(2^n-1)}\!\oplus \Z_{g_4}^{r_{n,|L|}-s_n-p(n+1)},	& \Sigma=0,g=0,a\neq 0,\\
	\Z_{\gcd(g,\Sigma)}^{s_n}\!\oplus \Z_{g_4}^{r_{n,|L|}-s_n-p(n+1)}, & \textrm{otherwise},\\
\end{cases}
$$
where $r_{n,k}=\frac{1}{1+k}\big(k^{n+1}+(-1)^n\big)$, $s_n=J(r_{n,2}-p(n+1))$
and $p(n)$ is the~parity of $n$.
\end{theorem}
\begin{proof}
If all $a,g$ and $\Sigma$ are zero, then $\partial=0$ and homology
groups are equal to the chain groups. Otherwise, notice first that $\partial^{\rt}$
vanishes on $\CF(L,t)$, so the~first formula, if $g_3=1$, follows directly from
the~formula~(\ref{eq:H-for-abc=1}) in Lemma~\ref{lem:H-split}.
For a~general case use Lemma~\ref{lem:diff-qdiff} to compute the~power
of $\Z_{g_3}$, knowing that $\rk\CF_{\!\!n} = |L|^n(|L|-1)$.

For the~second formula, assume first that $g_4=1$.
Proposition~\ref{prop:F0-hom} gives then recursive formulas for
the~size of homology, with a~solution being either $J(2^n-1)$
or $J(r_{n,2}-p(n+1))$ in this case. Then, again we use Lemma~\ref{lem:diff-qdiff}
to compute the~power of $\Z_{g_4}$ in a~general case, having
$\rk\F_{\!\!n}(L,t) = |L|^n-1$.
\end{proof}

The~theorem above answers Conjecture 29 stated in \cite{P-S}.
The~case of $(a,b,c,d)=(1,-1,0,0)$ was also conjectured by
S.~Carter. Recall that the~Boolean algebra modeled on a~set with $J$ elements
is denoted by $B_J$.

\begin{corollary}
Consider the~four-term augmented homology for $B_J$ with scalars $(a,b,c,0)$.
Then
$$
\rk H_n(B_J) = \begin{cases}
	|B_J|^{n+1} = 2^{J(n+1)},	& (a,b,c)=(0,0,0),\\
	J2^n+\delta_{0,n},			& (a,b,c)=(a,-a,-a),a\neq 0,\\
	1, 								& a+b+c=0,a\neq 0,\\
	\delta_{0,n},					& \textrm{otherwise},
\end{cases}
$$
where $\delta_{0,n}$ is $1$ if $n=0$ and $0$ otherwise.
\end{corollary}

We will end this section with computing normalized homology groups.
Again, it is enough to make computation for $B_1$.
Because each $C_n^N(B_1)\subset\CF(B_1)$ has only two elements,
the~argument from Proposition~\ref{prop:F0-C-hom} implies
$
H_n^N(B_1)=\Z^2
$
if $g=0$ and otherwise
$$
H_n^N(B_1)=\begin{cases}
	\Z_g\oplus\Z, & n=0,\\
	\Z_g,         & n>0.
\end{cases}
$$
If $L$ is finite then $|C^N_n(L)|=|L|(|L|-1)^n$ and a~similar
reasoning as before gives the~following.

\begin{theorem}\label{thm:hom-norm}
Let $L$ be a~finite distributive lattice.
Denote by $J$ the~number of non-minimal irreducible
elements in $L$ and put $g=\gcd(a+b,a+c)$ and $g_3=\gcd(a,b,c)$. Then
\begin{itemize}
\item if $g=0$ and $a=0$,
$
	H_n^N(L) = \Z^{|L|(|L|-1)^n}
$
\item if $g=0$ but $a\neq 0$,
$$
H_n^N(L) = \begin{cases}
	\Z^{J+1}\oplus\Z_a^{|L|-1-J},									& n=0,\\
	\phantom{{}^1}\Z^{2J}\oplus \Z_a^{(|L|-1)^{n+1}-J},	& n>0,
\end{cases}
$$
\item if $g\neq 0$,
$$
H_n^N(L)=\begin{cases}
	\Z_g^J\oplus \Z_{g_3}^{|L|-1-J} \oplus\Z,		& n=0,\\
	\Z_g^J\oplus \Z_{g_3}^{(|L|-1)^{n+1}-J},		& n>0.
\end{cases}
$$
\end{itemize}
\end{theorem}

\section{Odds and ends}\label{chpt:ends}
\subsection{Semi-lattices and spindles}
Part of the~results of this paper can be applied to semi-lattices and even to
spindles. For example, taking $c=0$ is equivalent to forgetting the~meet operation
and considering the~lattice as a~semi-lattice.%
\footnote{
	In this way, we obtain a~chain complex and homology for any lattice,
	not only a~distributive one.
}
In particular, Proposition~\ref{prop:ac-ab-ann-H} can be restated as follows.

\begin{proposition}
Let $(C,\partial)$ be the~three-term chain complex associated to a~spindle
$(X,\star)$ for scalars $(a,b,d)$. Then
\begin{itemize}
\item if $X$ has a~right projector $p$, then $a$ annihilates $H(X,p)$,
\item if $X$ has a~right unit $u$, then $(a+b)$ annihilates $H(X,u)$.
\end{itemize}
\end{proposition}

A~right \emph{projector} is an element $p$ such that $x\star p = p$
for any $x$. If $L$ is a~complete lattice, then a~semi-lattice $(L,\lor)$
has as a~projector the~maximal element $\top$ and as a~unit
the~minimal element $\bot$.

As a~result, if a~semi-lattice has both a~projector and a~unit,
then its~distributive homology is almost trivial. This is due to the~analogue
of Lemma~\ref{lem:H-split}.

\begin{lemma}
Let $(X,\star)$ be a~spindle with a~right unit $u$ and a~right projector $p$.
Assume $\gcd(a,b)=1$. Then
\begin{equation}
H_n(X,p) = H_n(\{p\},p) = 0.
\end{equation}
\end{lemma}

\noindent The~proof follows the~one of Lemma~\ref{lem:H-split} taking $y=p$.

In a~general case, we have to add a~number of copies of $\Z_g$ and $\Z_{\gcd(g,\Sigma)}$,
where as usual $g=\gcd(a,a+b)=\gcd(a,b)$ and $\Sigma=a+b+d$. Notice, that
$\gcd(g,\Sigma)=\gcd(a,b,d)$. The~following proposition is proven in the~same
way as Theorem~\ref{thm:hom-all}.

\begin{proposition}\label{prop:semi-lattice-hom}
Let $(X,\star)$ be a~finite spindle with a~right projector $p$ and a~right unit element $u$.
Consider the~three-term chain complex $C(X)$ for scalars $(a,b,d)$ and put
$g=\gcd(a,b)$ and $\Sigma=a+b+d$. Then
$$
H_n(X,p) = \begin{cases}
	\phantom{{}^{.+s_n}}\Z^{|X|^{n+1}-1},
			& \Sigma=0,g=0,\\
	\phantom{{}^{.+s_n}}\Z^{|X|^n(|X|-1)} \oplus \Z_{\Sigma}^{r_{n,|X|}-p(n+1)},
			& \Sigma\neq 0,g=0,\\
	\Z_g^{(|X|-1)r_{n,|X|}+s_n} \oplus \Z_{\gcd(a,b,d)}^{r_{n,|X|}-p(n+1)-s_n},
			& \Sigma=0,g\neq 0,\\
	\phantom{{}^{..s_n}}\Z_g^{(|X|-1)r_{n,|X|}} \oplus \Z_{\gcd(a,b,d)}^{r_{n,|X|}-p(n+1)},
			& \Sigma\neq 0,g\neq 0,
\end{cases}
$$
where $r_{n,k}$ and $s_n$ are defined as in Theorem~\ref{thm:hom-all}
and $p(n)$ stands for the~parity of $n$.
\end{proposition}

Every finite semi-lattice has a~projector but only a~few have units.
For example, a~rooted tree has a~unit element only if it is a~chain.
The~authors computed homology groups for all semi-lattices on a~set with up to
four elements. It revealed that all of them, even those with no unit elements,
follow the~pattern from Proposition~\ref{prop:semi-lattice-hom}.

\begin{conjecture}
Let $L$ be a~semi-lattice with a projector and assume $\gcd(a,b)=1$.
Then $H_n(L,t)=0$.
\end{conjecture}

\begin{table}[tb]
\begin{tabular}{l|cccc}
\hline
\hline
spindle type     & any		& idempotent & associative & commutative \\
\hline
all spindles     & 185 (41)&  94 (14)  &  47 (13)  &  10  \\
with a~unit      &  82 \phantom{(0)} &  62 \phantom{(10)}   &  20 \phantom{(10)}   &  \phantom{1}3  \\
with a~projector & \phantom{8}78 (38)&  40 (12)  &  36 (12)  & \phantom{1}8 \\
with both        &  19 \phantom{(0)}	&  17 \phantom{(10)}   &  13 \phantom{(10)}   &  \phantom{1}3  \\
with none        &  44 (3)	&   9 (2)   &   4 (1)   &  \phantom{1}2  \\
\hline
\end{tabular}

\vskip 12pt
\caption{
	Spindles $(X,\star)$ with $3\leqslant |X| \leqslant 4$.
	Numbers in brackets, if present, indicate spindles for which
	Theorem~\ref{thm:spindle-orbit} does not hold.%
}\label{tbl:spindles-stat}
\end{table}

The~conjecture is not true for spindles in general. The~simplest example
is given by the~\projleft{} operation $x\lt y=x$. It is easy to see,
that in this case $H(\CF(X,t))$ consists of copies of $\Z_{a+b}$.
Since every element is a~right unit, each orbit is equal to $X$.
However, computation showed that homology for a~spindle $X$ usually does not
change when replacing $X$ with any of its right orbit $X\star x$.
When $X$ has up to four elements, this is true provided the~operation
is unital or commutative (see tab.~\ref{tbl:spindles-stat}).
The~first case is easy to prove.

\begin{theorem}\label{thm:spindle-orbit}
Let $(X,\star)$ be a~spindle with a~right unit
and assume $\gcd(a,b)=1$. Then for any $x\in X$ and $t\in X$
\begin{equation}
	H(X,t) = H(X\star x, t\star x).
\end{equation}
\end{theorem}
\begin{proof}
Without loss of generality, we can assume $t=x$.
The~exact sequence (\ref{eq:seq-subcomplex}) with $A=X\star x$
does not decompose as in Proposition~\ref{prop:compl-split},
unless $\star\star=\star$. However, it induces a~long exact sequence of homologies
$$
\ldots\to H_n(X\star x, x)\to H_n(X,x)\to H_n(X,X\star x)\to H_{n-1}(X\star x, x)\to\ldots
$$
and it suffices to show that $H(X,X\star x)$ vanishes.

Define a~chain map $f\colon C(X,x)\to C(X,x)$ by
$f(y_0,\dots,y_n) = (y_0\star x,\dots,y_n\star x)$.
It takes values in $C(X\star x,x)$, so that it induces a~trivial map on $C(X,X\star x)$.
Since $X$ is a~spindle, $x=x\star x\in X\star x$ and a~homotopy $h^x$ defined
as in~the~proof of Lemma~\ref{lem:H-split} induces homotopies on both $C(X\star x, x)$
and $C(X,X\star x)$ between $a\id$ and $bf$. This shows that $a$ annihilates $H(X,X\star x)$.
Existence of a~right unit provides that $(a+b)$ annihilates
homology as well and as a~result $H(X, X\star x)=0$, since $\gcd(a,b)=1$.
\end{proof}

The~theorem says nothing about homology of quandles, however.
This is because quandles have no proper right orbits,
what means that they are the~basic pieces for spindles from the~point of
view of distributive homology. On the~other hand, if we restrict
to idempotent spindles (i.e. those with $\star\star=\star$),
then the~only operation with no proper right orbits is $\lt$,
whose homology is easy to compute: it is acyclic if $a+b=1$ and otherwise
it follows from Lemma~\ref{lem:diff-qdiff}.
Therefore, Theorem~\ref{thm:spindle-orbit} gives a~full answer
for those spindles.

We conjecture the~theorem holds also for all commutative spindles.
These are more general than semi-lattices, because commutativity does
not imply associativity%
\footnote{
	However, $X$ is associative if $\star$ is also an~idempotent in $Bin(X)$,
	i.e. $\star\star = \star$.
}
and the~smallest non-associative example is provided by the~dihedral
quandle $R_3$ with three elements, where $X=\Z_3$ and $x\star y = 2y-x$.

\begin{conjecture}\label{conj:3-monoids}
If $(X,\star)$ is a~commutative spindle, $t,x\in X$ and $\gcd(a,b)=1$,
then $H_n(X,t) = H_n(X\star x,t\star x)$.
\end{conjecture}

\subsection{Skew lattices}
Methods of this paper extend nicely over \emph{skew lattices}, which are
noncommutative variants of lattices. Here we state basic definitions and facts.
For more details and proofs the~reader is referred to \cite{LeechRec,Leech}.

A~skew lattice is an~algebraic system $(L,\land,\lor)$ with both operations being
associative and idempotent, such that all variants of absorption law hold:
\begin{align}
x\land (x\lor y) =&\ x = (y\lor x)\land x, \\
x\lor (x\land y) =&\ x = (y\land x)\lor x.
\end{align}
\noindent There is a~natural partial order on $L$
\begin{equation}
	x\leq y \quad\equiv\quad x\land y = x = y\land x \quad\equiv\quad x\lor y = y = y\lor x
\end{equation}
and a~natural pre-order%
\footnote{
	A~pre-order is transitive and reflexive, but might not be antisymmetric.
}
\begin{equation}
	x\preccurlyeq y \quad\equiv\quad x\land y\land x = x \quad\equiv\quad x\lor y\lor x = y.
\end{equation}
The latter induces an~equivalence relation
\begin{equation}
	x\sim y\quad\equiv\quad x\preccurlyeq y \textrm{ and } y\preccurlyeq x
\end{equation}
and the~equivalence classes are called $\mathcal{D}$-classes.
As in case of lattices, we can create Hasse diagrams for skew lattices (fig.~\ref{fig:skew}).
However, contrary to lattices, they do not include enough information to read
the~full structure of a~skew lattice.

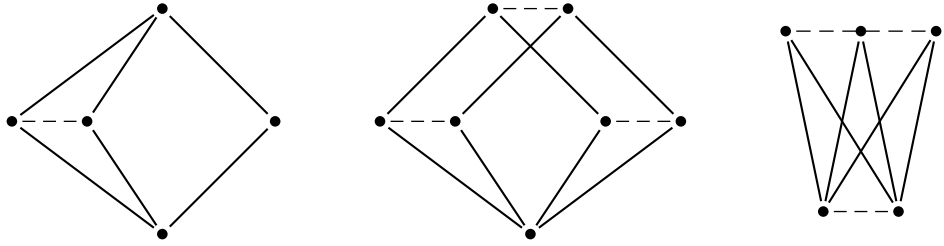
\begin{figure}[t]
	\begin{center}
		\hfill\hfill
\begin{pspicture}(3.5,3)
\cnode*(2,0){2pt}{b}
\cnode*(0,1.5){2pt}{l1}
\cnode*(1,1.5){2pt}{l2}
\cnode*(3.5,1.5){2pt}{r}
\cnode*(2,3){2pt}{t}
\ncline[nodesep=2pt,linestyle=dashed,linewidth=0.5pt]{l1}{l2}
\ncline[nodesep=2pt]{b}{l1}
\ncline[nodesep=2pt]{b}{l2}
\ncline[nodesep=2pt]{b}{r}
\ncline[nodesep=2pt]{l1}{t}
\ncline[nodesep=2pt]{l2}{t}
\ncline[nodesep=2pt]{r}{t}
\end{pspicture}
\hfill
\begin{pspicture}(4,3)
\cnode*(2,0){2pt}{b}
\cnode*(0,1.5){2pt}{l1}
\cnode*(1,1.5){2pt}{l2}
\cnode*(3,1.5){2pt}{r1}
\cnode*(4,1.5){2pt}{r2}
\cnode*(1.5,3){2pt}{t1}
\cnode*(2.5,3){2pt}{t2}
\ncline[nodesep=2pt,linestyle=dashed,linewidth=0.5pt]{l1}{l2}
\ncline[nodesep=2pt,linestyle=dashed,linewidth=0.5pt]{r1}{r2}
\ncline[nodesep=2pt,linestyle=dashed,linewidth=0.5pt]{t1}{t2}
\ncline[nodesep=2pt]{b}{l1}
\ncline[nodesep=2pt]{b}{l2}
\ncline[nodesep=2pt]{b}{r1}
\ncline[nodesep=2pt]{b}{r2}
\ncline[nodesep=2pt]{l1}{t1}
\ncline[nodesep=2pt]{l2}{t2}
\ncline[nodesep=2pt]{r1}{t1}
\ncline[nodesep=2pt]{r2}{t2}
\end{pspicture}
\hfill
\begin{pspicture}(2,2)
\cnode*(0.5,0.3){2pt}{b1}
\cnode*(1.5,0.3){2pt}{b2}
\cnode*(0,2.7){2pt}{t1}
\cnode*(1,2.7){2pt}{t2}
\cnode*(2,2.7){2pt}{t3}
\ncline[nodesep=2pt,linestyle=dashed,linewidth=0.5pt]{b1}{b2}
\ncline[nodesep=2pt,linestyle=dashed,linewidth=0.5pt]{t1}{t3}
\ncline[nodesep=2pt]{b1}{t1}
\ncline[nodesep=2pt]{b1}{t2}
\ncline[nodesep=2pt]{b1}{t3}
\ncline[nodesep=2pt]{b2}{t1}
\ncline[nodesep=2pt]{b2}{t2}
\ncline[nodesep=2pt]{b2}{t3}
\end{pspicture}
\hfill\hfill\ 
	\end{center}
	\caption{Hasse diagrams of skew lattices. Solid lines represent the~natural partial
				order, while dashed lines group elements into $\mathcal{D}$-classes.}
	\label{fig:skew}
\end{figure}

Notice, that for any skew lattice $L$, the~quotient $L/\!\!\sim$ is a~lattice.
The~Clifford-McLean Theorem (see \cite{LeechRec}) says this is the~largest lattice
among quotients of $L$. The~quotient map is not always a~retraction, but this is
the~case of $L$ is \emph{symmetric}, that is $x\land y = y\land x$ if and only if
$x\lor y = y\lor x$.

Naive distributivity conditions are too strong, because they imply a~skew lattice
is a~lattice. Therefore, a~weaker condition is imposed. Say that a~skew lattice
$(L,\land,\lor)$ is \emph{distributive}, if the~operations are distributive
from both sides at the same time:
\begin{align}
	x\land(y\lor z)\land x &= (x\land y\land x)\lor (x\land z\land x), \\
	x\lor(y\land z)\lor x &= (x\lor y\lor x)\land (x\lor z\lor x).
\end{align}
This does not imply distributivity in our sense.
However, we can define ``conjugated'' operations $\triangledown$ and $\triangleup$
as follows:
$$
	x\triangledown y := y\lor x\lor y \qquad\qquad
	x\triangleup y := y\land x\land x.
$$
It is a~straightforward calculation to check that these operations are idempotent,
distributive%
\footnote{
	Self-distributivity is a~bit tricky. It requires \emph{regularity},
	i.e. $x\star y\star x\star z\star x = x\star y\star z\star x$, and is proven
	in \cite{LeechRec}.
}
with respect to each other (if $L$ is distributive) and they satisfy
the~absorption law. Moreover, they are idempotent elements in $Bin(L)$, i.e.
$(x\triangleup y)\triangleup y = x\triangleup y$ and similarly for $\triangledown$.
If $L$ is a~lattice then the~conjugated operations are equal to the~original one.
Notice, that $\triangledown$ and $\triangleup$ are still associative.

Most of the~results from this paper extend to skew lattices.
The~first observation is that the~quotient map $L\to L/\!\!\sim$
induces a~map on homology groups. If $L$ is~symmetric, we have more.

\begin{corollary}
If $L$ is a~symmetric skew lattice, then $H_n(L/\!\!\sim)$ is a~direct summand of $H_n(L)$.
\end{corollary}

\noindent The~simplest example with $H(L)$ being strictly larger is provided
by a~\emph{rectangular} skew lattice, that is a~skew lattice consisting
of a~unique $\mathcal{D}$-class. In this case, $\triangledown$ coincides
with $\lt$ and $\triangleup$ with $\rt$, and homology groups contain
$\Z_{a+b}$ as summands.

If a~skew lattice $L$ has a~minimum and a~maximum, then all proofs
can be repeated. In particular, we can show the~following.

\begin{proposition}
Let $L$ be a~finite skew lattice with a~unique maximum and a~unique minimum.
Then $H_n(L) \cong H_n(L/\!\!\sim)\oplus\Z_{\gcd(a,b,c)}^p\oplus\Z_{\gcd(a,b,c,d)}^q$
for some $p$ and $q$.
\end{proposition}

\noindent This is because any chain in $L$ induces a~chain of the~same length in $L/\!\!\sim$.
The~numbers $p$ and $q$ reflect the~difference in size of complexes and are easy to compute.
More mysterious are computation that shows the~above proposition holds
if a~skew lattice has only maximum or only minimum.%
\footnote{
	This can happen if the~$\mathcal{D}$-class of $\top$ or $\bot$ has more than one element.
}
The~following conjecture has been checked for all skew-lattices with at most four elements.

\begin{conjecture}
If $L$ is a~finite skew lattice with either a~unique minimum or a~unique maximum,
$H_n(L) \cong H_n(L/\!\!\sim)\oplus\Z_{\gcd(a,b,c)}^p\!\oplus\Z_{\gcd(a,b,c,d)}^q$
for some $p$ and $q$.
\end{conjecture}

On the~other hand, the~example of a~rectangular skew lattice shows that at least
a~minimum or a~maximum must exists.

\subsection{Mayer-Vietoris sequence for 2-spindles}

Theorem~\ref{thm:spindle-orbit} extends in an~interesting way to multisplindles.
If $X$ is equipped with two operations, it gives an~exact sequence of homology
groups similar to the~Mayer-Vietoris sequence known in algebraic topology.
It follows from the~Lemma below, which is a~generalization
of Lemma~\ref{lem:H-split}.

Let $(X,\star_1,\star_2)$ be a~multisplindle and $x\in X$. Denote by
$\mathcal{O}_i:=X\star_i x$ the~orbit of $x$ with respect to $\star_i$.
Let $C(X,\mathcal{O}_1+\mathcal{O}_2)$ be the~quotient
\begin{equation}
	C(X,\mathcal{O}_1,\mathcal{O}_2):=\frac{C(X)}{C(\mathcal{O}_1)+C(\mathcal{O}_2)}
\end{equation}
and denote its homology by $H(X,\mathcal{O}_1,\mathcal{O}_2)$.
As usual we denote the~scalars involved in construction of $C(X)$ by $a,b,c,d$.

\begin{lemma}\label{lem:H(X,A)-vanishes}
Homology $H(X,\mathcal{O}_1,\mathcal{O}_2)$ is annihilated by $a$.
\end{lemma}
\begin{proof}
A~chain homotopy $h^x(w_0,\dots,w_n):=(w_0,\dots,w_n,x)$ from Lemma~\ref{lem:H-split}
induces on $C(X,\mathcal{O}_1,\mathcal{O}_2)$ a~homotopy between $a\cdot\id$ and
a~trivial map, because for any $w\in C(X)$ all elements $w\star_1 x$, $w\star_2 x$
and $w\rt x$ lie in $C(\mathcal{O}_1)+C(\mathcal{O}_2)$.
\end{proof}

\begin{theorem}\label{thm:MVSeq}
Let $(X,\star_1,\star_2)$ be a~multispindle and $x\in X$.
If the~coefficient $a$ is invertible, then there is a~long exact sequence
$$
\dots\to H_{n+1}(X)\to
H_n(\mathcal{O}_1\cap\mathcal{O}_2)\to
H_n(\mathcal{O}_1)\oplus H_n(\mathcal{O}_2)\stackrel{(i_*,-j_*)}\to
H_n(X)\to \dots
$$
where $i_*$ and $j_*$ are induced by inclusions.
\end{theorem}
\begin{proof}
The~proof goes in two steps. First, there is a~standard exact sequence
of complexes
$$
0\to C(\mathcal{O}_1\cap\mathcal{O}_2)\to
C(\mathcal{O}_1)\oplus C(\mathcal{O}_2)\to
C(\mathcal{O}_1)+C(\mathcal{O}_2)\to 0,
$$
which results in a~long exact sequence of homology we want, expect that
every $H_n(X)$ is replaced with $H_n(C(\mathcal{O}_1)+C(\mathcal{O}_2))$.
Next, use the~definition of $H(X,\mathcal{O}_1,\mathcal{O}_2)$
and Lemma~\ref{lem:H(X,A)-vanishes} to show that those homology groups
are equal if $a$ is invertible.
\end{proof}

Assumptions of the~theorem above are stronger than saying that $\gcd(a,b,c)=1$,
because we do not know, whether $b$ and $c$ annihilates 
$H(C(\mathcal{O}_1)+C(\mathcal{O}_2))$ as well. This holds, however,
if each operation has a~unit that is at the~same time a~projector for
the~other operation.

\subsection{Multishelves with absorption}\label{sec:mult-w-abs}

One of the~most important properties of a~lattice is absorption~(\ref{eq:absorption}).
It is used in this paper in many places, so it seems natural to ask
what it implies for multishelves.

\begin{definition}
Say that a~multishelf $(X,\{\star_\lambda\}_{\lambda\in\Lambda})$
\emph{satisfies the~absorption law}, if for any two operations
$\star_\alpha\neq\star_\beta$:
\begin{equation}\label{eq:msh-abs}
(x\star_\alpha y)\star_\beta y = y.
\end{equation}
\end{definition} 

It is a~classical observation \cite{LeechRec} that a~multishelf with absorption
must be a~multispindle. This is an~analogue to the~statement, that idempotency follows
from other axioms of a~lattice. Hence, from the~point of view of distributive
homology, we can add both the~left and the~right trivial operation as it was
in the~case of lattices.

\begin{proposition}
Let $(X,\{\star_\lambda\}_{\lambda\in\Lambda})$ be a~multishelf with absorption.
Then $x\star_\lambda x = x$ for any $\lambda\in\Lambda$.
\end{proposition}
\begin{proof}
Putting $x=y$ or $x=y\star_\beta y$ in (\ref{eq:msh-abs}) we obtain
$y\star_\beta y = ((y\star_\beta y)\star_\alpha y)\star_\beta y = y$.
\end{proof}

If a~multishelf with absorption consists only of idempotent elements of $Bin(X)$,
that is $(x\star_\lambda y)\star_\lambda y = x\star_\lambda y$ for every $\lambda\in\Lambda$,
then we call $X$ a~\emph{generalized distributive lattice}. Moreover, if we relax
the~conditions and do not assume the~operations are mutually distributive,
then $X$ is called a~\emph{generalized lattice}.

To have analogous results to Proposition~\ref{prop:ac-ab-ann-H} we need elements
that behaves likes $\top$ and $\bot$. Because these are neutral elements for
$\land$ and $\lor$ respectively, we introduce the~following definition.

\begin{definition}
A~multishelf $(X,\{\star_\lambda\}_{\lambda\in\Lambda})$ is \emph{unital},
if every $\star_\lambda$ has a~right unit $u_\lambda$.
\end{definition}

The~smallest nontrivial example is given by a~set $B_1^k=\{u_1,\dots,u_k\}$
with operations $\star_1,\dots,\star_k$ defined as
$u_i\star_s u_s = u_i$ and $u_i\star_s u_j = u_j$ for $j\neq s$.
It is the~smallest generalized unital distributive lattice consisting of
$k$ operations and at least two elements.

It is worth to notice, that in a~unital multishelf with absorption units
$u_\lambda$ are automatically projectors for other operations:
\begin{equation}
x\star_\beta u_\alpha = (x\star_\alpha u_\alpha)\star_\beta u_\alpha = u_\alpha.
\end{equation}
As a~consequence, for a~unital multishelf with absorption we do not have to assume
$a$ to be invertible in Theorem~\ref{thm:MVSeq}.
Furthermore, an~intersection of orbits has only one element.
Indeed, if $x\in \orb_i\cap \orb_j$ for $i\neq j$, then $x = y\star_i t$
and $x=y'\star_j t$ for some $y$ and $y'$, what implies
\begin{equation}
x = (y'\star_j t)=(y'\star_j t)\star_j t = x\star_j t = (y\star_i t)\star_j t = t.
\end{equation}
As a~result, the~sequence from Theorem~\ref{thm:MVSeq} shows that
reduced homology splits as in Lemma~\ref{lem:H-split}.
This generalizes in an~obvious way to multispindles with more operations
satisfying the~absorption law.

The~remark above shows that as long as we see an~element $x\in X$ with all
orbits strictly smaller that $X$, we can reduce computation of homology to
smaller multispindles. If we cannot, that is for every element $x\in X$ at least
one of the~orbits $X\star_i x$ equals the~whole set $X$, we call the~multishelf
$X$ \emph{irreducible}.
In case of distributive lattices the~only the~Boolean algebra $B_1$ is irreducible.
For multishelves with absorption there are more cases and
all are characterized by a~simple condition.

\begin{proposition}
Let $(X,\{\star_1,\dots,\star_k\})$ be a~multishelf with absorption.
Then $X$ is irreducible if and only if every
element $x\in X$ is a~unit for some operation $\star_i$.
\end{proposition}
\begin{proof}
The ``if'' part is easy. For the ``only if'' pick any $u\in X$
and notice that if $X$ is irreducible, then for some operation
$\star_i$ the~map $x \mapsto x\star_i u$ must be a~bijection.
In particular, $x = y\star_i u$ for some $y$.
But since $\star_i$ is an~idempotent in $Bin(X)$, we have also
$$
x\star_i u = (y\star_i u)\star_i u = y\star_i u = x,
$$
what proves that $u$ is a~unit for $\star_i$.
\end{proof}

In particular, every irreducible multishelf with absorption is a~generalized lattice.
Denote by $I^{(r_1,\dots,r_k)}$ the~irreducible generalized lattice
with exactly $r_i$ units for the~$i^{\mathrm{th}}$ operation.
Its homology groups can be computed in the~same
way as we did for $B_1$. Indeed, when restricted to $\CF$, the~differential
is divisible by $g=\gcd\{a_0+a_i\ |\ \star_i\textrm{ has a unit}\}$
and existence of units guarantees that each such $a_0+a_i$ annihilates homology.
Hence, only copies of $\Z_g$ appears in $H\CF$.
Now we can use~Proposition~\ref{prop:F0-hom} to compute the~whole groups.

As a~result, we obtain an~efficient algorithm to compute homology groups
for any finite unital multishelf with absorption.

\begin{algorithm}\label{algorithm}
Let $X$ be a~finite multishelf with absorption. If $X$ has units,
reduced homology $H(\CF(X,x_0))$ can be computed by the~following steps.
\begin{enumerate}[label={\normalfont\arabic{*}. }]
	\item Search for $t\in X$ with all orbits $\orb_i$ being proper subsets of $X$.
	\item If such $t$ does not exist, $X$ is irreducible and we already know its homology.
	\item If such $t$ exists, compute homology groups for each orbit $\orb_i$
		and glue them together as in Lemma~\ref{lem:H-split}.
\end{enumerate}
\end{algorithm}

\begin{table}[bt]
\begin{tabular}{l|ccc}
\hline
\hline
multishelf type  & $\quad$lattice$\quad$ & skew lattice & gen. lattice \\
\hline
All monoids  &  3  & 56 (15)           & 191 (74)\phantom{1}  \\
2U + 2P      &  3  &  4 \phantom{(0)}  &  31 \phantom{(01)}  \\
1U + 2P      &  0  & 16 \phantom{(01)} &  60 (37)  \\
1U + 1P      &  0  &  7 \phantom{(0)}  &  46 (1)\phantom{0}  \\
0U + 2P      &  0  & 25 (14)           &  43 (32)  \\
0U + 1P      &  0  &  4 (1)            &  10 (3)\phantom{1}  \\
0U + 0P      &  0  &  0 \phantom{(0)}  &   1 (1)  \\
\hline
\end{tabular}

\vskip 12pt
\caption{
	Multishelves with absorption $(X,\star_1,\star_2)$ with $3\leqslant |X| \leqslant 4$
	up to duality. The notation ``1U+2P'' indicates multishelves where only one operation
	has a~unit, but both have projectors. Similarly the others. Numbers in brackets,
	if present, count monoids for which Algorithm~\ref{algorithm} fails.
}\label{tbl:4-monoids-stat}
\end{table}

\noindent What remains to show is that all the~orbits, if reducible, have units.
However, the~first guess that $u_s\star_i t\in \orb_i$ is a~unit for $\star_s$,
where $u_s$ is a~unit for $\star_s$ in $X$, is correct:
\begin{equation}
	(x\star_i t)\star_s (u_s\star_i t) = (x\star_s u_s)\star_i t = x\star_i t.
\end{equation}
Therefore, we can continue the~reduction of orbits $\orb_i$ until we end
with irreducibles, what proves correctness of the~algorithm.

The~authors computed homology for all generalized distributive
lattices with two operations on a~set with up to four elements
(see tab.~\ref{tbl:4-monoids-stat}). It appeared, that
Algorithm~\ref{algorithm} still might be correct for skew lattices
with at least one operation being unital, but nothing more.

In general, the~homology is much richer than the~one predicted
by Algorithm~\ref{algorithm}. Below we present $H(\CF(X,t))$ for 
scalars $(4,5,2,0)$ for two such examples $X$.
In addition to the~torsion already known, we also observe
$\Z_a, \Z_{a+b}$ and $\Z_{\gcd(a,c)}$.

\vskip 12pt
\begin{center}
\begin{tabular}{|cc|l|l|}
\hline
\multicolumn{2}{|c|}{\textbf{Generalized}} &
\multicolumn{2}{c|}{\textbf{Homology}} \\
\cline{3-4}
\multicolumn{2}{|c|}{\textbf{distributive lattice}} &
\multicolumn{1}{c}{predicted} & \multicolumn{1}{|c|}{actual} \\
\hline
\multirow{4}{*}{\scriptsize\begin{tabular}{c|cccc}
$\star_1$ & 1 & 2 & 3 & 4 \\
\hline
1 & 1 & 2 & 1 & 2 \\
2 & 1 & 2 & 1 & 2 \\
3 & 1 & 2 & 3 & 4 \\
4 & 1 & 2 & 3 & 4
\end{tabular}}
&
\multirow{4}{*}{\scriptsize\begin{tabular}{c|cccc}
$\star_2$ & 1 & 2 & 3 & 4 \\
\hline
1 & 1 & 2 & 3 & 4 \\
2 & 2 & 2 & 3 & 4 \\
3 & 1 & 2 & 3 & 4 \\
4 & 2 & 2 & 3 & 4
\end{tabular}}

 &$H_0=0\phantom{\Big|}$ & $H_0 = \Z_2$                     \\
&&$H_1=0$                & $H_1 = \Z_2^4\oplus\Z_4$         \\
&&$H_2=0\phantom{\Big|}$ & $H_2 = \Z_2^{10}\oplus\Z_4^3$    \\
&&$H_3=0$                & $H_3 = \Z_2^{38}\oplus\Z_4^{13}$ \\

\hline

\multirow{4}{*}{\scriptsize\begin{tabular}{c|cccc}
$\star_1$ & 1 & 2 & 3 & 4 \\
\hline
1 & 1 & 2 & 1 & 2 \\
2 & 1 & 2 & 1 & 2 \\
3 & 3 & 4 & 3 & 4 \\
4 & 3 & 4 & 3 & 4
\end{tabular}}
&
\multirow{4}{*}{\scriptsize\begin{tabular}{c|cccc}
$\star_2$ & 1 & 2 & 3 & 4 \\
\hline
1 & 1 & 1 & 3 & 3 \\
2 & 2 & 2 & 4 & 4 \\
3 & 1 & 1 & 3 & 3 \\
4 & 2 & 2 & 4 & 4
\end{tabular}}
 &$H_0=\Z_3^2\phantom{\Big|}$ & $H_0 = \Z_2\oplus\Z_3\oplus\Z_9$                         \\
&&$H_1=\Z_3^2$                & $H_1 = \Z_2^2\oplus\Z_3\oplus\Z_4\oplus\Z_9$             \\
&&$H_2=\Z_3^6\phantom{\Big|}$ & $H_2 = \Z_2^{10}\oplus\Z_3^3\oplus\Z_4^3\oplus\Z_9^3$    \\
&&$H_3=\Z_3^{10}$             & $H_3 = \Z_2^{38}\oplus\Z_3^5\oplus\Z_4^{13}\oplus\Z_9^5$ \\
\hline
\end{tabular}
\end{center}

\ifnum\value{countcomments} > 0
\message{LaTeX Warning: \thecountcomments\ comments remained to check!}
\fi

\end{document}